\documentclass{amsart}

\usepackage{amsmath,amsfonts,amssymb,amscd,verbatim}

\def\Tr{\mathrm{Tr}}

                                              \def\Norm{\mathrm{Norm}}
                                                   \def\Tr{\mathrm{Tr}}

\def\fchar{\mathrm{char}}

                                                    \def\rr{\mathbf{r}}

\newtheorem{thm}{Theorem}[section]

\newtheorem{cor}[thm]{Corollary}
\newtheorem{prop}[thm]{Proposition}

\theoremstyle{definition}

\newtheorem{ex}[thm]{Example}

\newtheorem{sect}[thm]{}

           \newtheorem{rem}[thm]{Remark}

\newtheorem{rems}[thm]{Remarks}

\title{Families of  elliptic curves with rational torsion points of even order}
\author {Boris M. Bekker}
\address{St. Petersburg State University, Department of Mathematics and Mechanics,   Universitetsky prospekt, 28, Peterhof, St. Petersburg, 198504, Russia}
\email{ bekker.boris@gmail.com}
\author {Yuri G. Zarhin}
\address{Pennsylvania State University, Department of Mathematics, University Park, PA 16802, USA}
\email{zarhin@math.psu.edu}

\begin{document}

\maketitle

\section{Introduction}

This paper is a follow up of \cite{BZAA}. Our aim is to give an explicit construction of versal families of elliptic curves with a torsion point of a certain (small) order. The problem of constructing such families goes back to Beppo Levi \cite{SS,Silver} and is closely related to certain modular curves of genus zero. However, our approach based on the explicit formulas for ``halves'' of points on elliptic curves \cite[Sect. 2]{BZAA} is quite elementary.

Here are our main results.

\begin{thm}
\label{04}
Let $K$ be a field with $\fchar{K}\ne 2$. Let $E$ be an elliptic curve over $K$.
\begin{enumerate}
\item
Let $a$ be a nonzero element of $K$.
The following conditions are equivalent.
\begin{itemize}
\item[(4i)]
$E(K)$ contains exactly one point of order $2$ and a point of order $4$.
\item[(4ii)]
There exists a nonzero element $b \in K$ such that $a^2+4b$ is not a square in $K$ and $E$ is $K$-isomorphic to the elliptic curve
$$\mathcal{E}^{(4)}_{a,b}: y^2=\left(x^2+(a^2+2b)x+b^2\right)x.$$
%with  $K$-point $(-b,ab)$ of order 4.
\end{itemize}
\item
The following conditions are equivalent.
\begin{itemize}
\item[(8i)]
$E(K)$ contains exactly one point of order $2$ and a point of order $8$.
\item[(8ii)]
There exists $t \in K\setminus \{0,\pm 1\}$ such that $2t^2-1$ is not a square in $K$ and $E$ is $K$-isomorphic to the elliptic curve
$$\mathcal{E}^{(8)}_t:  y^2=\left(x^2+2\frac{t^4+2t^2-1}{ (t^2-1)^2}x+1\right)x.$$
%with $K$-point ????? of order 8.
\end{itemize}
\item
The following conditions are equivalent.
\begin{itemize}
\item[(6i)]
$E(K)$ contains exactly one point of order $2$ and a point of order $6$.
\item[(6ii)]
There exists $t \in K \setminus \{0,-4, 1/2\}$ such that $t^2+4t$ is not a square in $K$ and $E$ is $K$-isomorphic to the elliptic curve
$$\mathcal{E}^{(6)}_t: y^2=(x^2+(t^2+2t)x+t^2)(x+1).$$
%with $K$-points $(0,t)$ of order 3 and $(-2t, -1+2t-2t^2)$ of order 6.
\end{itemize}
\item
The following conditions are equivalent.
\begin{itemize}
\item[(12i)]
$E(K)$ contains exactly one point of order $2$ and a point of order $12$.
\item[(12ii)]
There exists  $t \in K\setminus \{0,\pm1,\pm\sqrt{-1}\}$ such that  $(t^2+1)(3t^2-1)$ is not a square in $K$, $3t^2+1\neq0$, and $E$ is $K$-isomorphic to the elliptic curve
$$\mathcal{E}^{(12)}_{t}: y^2=(x+1)\left(x^2+\frac{8t^2(t^2+1)(t^4+4t^2-1)}{(t^2-1)^4}x+\frac{16t^4(t^2+ 1)^2}{(t^2-1)^4}\right).$$
 %with $K$-points $(0,t)$ of order 3 and $(-2t, -1+2t-2t^2)$ of order 6.
\end{itemize}
\item
The following conditions are equivalent.
\begin{itemize}
\item[(10i)]
$E(K)$ contains exactly one point of order $2$ and a point of order $10$.
\item[(10ii)]
There exists  $t \in K\setminus\{0,\pm1, (-1\pm\sqrt 5)/2, 2\pm\sqrt5\}$ such that  $t(t^2+t-1)$ is not a square in $K$ and $E$ is $K$-isomorphic  to the elliptic curve
$$\mathcal{E}^{(10)}_{t}: y^2=(x+1)\left(x^2+\dfrac{8t^2(t^3+t^2-t+1)}
 {(t-1)^2(t+1)^4}x+\frac{16t^4}{(t-1)^2(t+1)^4}\right).$$
 %with $K$-points $(0,t)$ of order 3 and $(-2t, -1+2t-2t^2)$ of order 6.
\end{itemize}
\end{enumerate}
\end{thm}

\begin{rems}
\begin{enumerate}
\item
$\mathcal{E}^{(4)}_{a,b}(K)$ contains exactly two points of order $4$, namely,
$(-b,ab)$  and $(-b,-ab)$. On the other hand, dividing both sides of the equation for $\mathcal{E}^{(4)}_{a,b}$
by $a^6$ and introducing
$$\tilde{x}=\frac{x}{a^2},\ \tilde{y}=\frac{y}{a^3}, \ \tilde{b}=\frac{b}{a^2},$$
we obtain that $\mathcal{E}^{(4)}_{a,b}$ is $K$-isomorphic to
$$\mathcal{E}^{(4)}_{1,\tilde{b}}: \tilde{y}^2=\left(\tilde{x}^2+(1+2\tilde{b})\tilde{x}+\tilde{b}^2\right)\tilde{x}.$$

\item
$\mathcal{E}^{(8)}_t(K)$ contains exactly two points of order 4, namely,
$$\left(1, \frac{2t}{1-t^2}\right), \ \left(1, -\frac{2t}{1-t^2}\right)$$
and exactly four points of order 8, namely,
$$\left(\frac{1+t}{1-t},\frac{-2t}{(1-t)^2}\right), \ \left(\frac{1-t}{1+t},\frac{2t}{(1+t)^2}\right)$$
and their negatives
$$\left(\frac{1+t}{1-t},\frac{2t}{(1-t)^2}\right), \ \left(\frac{1-t}{1+t},\frac{-2t}{(1+t)^2}\right).$$
\item
$\mathcal{E}^{(6)}_t(K)$ contains the points $(0,t)$ and   $(0,-t)$ of order 3  and the points $(-2t, t-2t^2)$ and $(-2t, 2t^2-t)$ of order 6.
 If $\fchar =3$, then there are no other points of order $3$ or $6$ in $\mathcal{E}^{(6)}_t(K)$.
 \item
 $\mathcal{E}^{(12)}_t(K)$
 contains the points $$\left(0,\frac{4t^2(t^2+1)}{(t^2-1)^2}\right), \  \left(0,-\frac{4t^2(t^2+1)}{(t^2-1)^2}\right)$$   of order 3, and exactly two points of order 4, namely,
$$\left(-\frac{3t^2+1}{t^2-1}, -\frac{8 t^3(3t^2+1)}{(t^2-1)^3}\right), \ \left(-\frac{3t^2+1}{t^2-1}, \frac{8 t^3(3t^2+1)}{(t^2-1)^3}\right).$$
 \item
  $\mathcal{E}^{(10)}_t(K)$ contains    a point $\left(0,{4t^2}/{(t-1)(t+1)^2}\right)$ of order $5$
and exactly one point $(-1,0)$ of order $2$.
\end{enumerate}
\end{rems}

The paper is organized as follows.  First three sections deal with division by 2 in $E(K)$
under various assumptions about the existence of $K$-points of order 2 on the elliptic curve $E$. Our goal is to obtain explicit formulas that will be used in the next five sections  containing a construction of versal families of elliptic curves with rational points of order 4,8,  6, 12, and 10,  respectively (and with exactly one rational point of order 2). The last section deals with versal families of (ordinary) elliptic curves in characteristic 2 that admit a rational point of order 4 or 8.

\section{Review of \cite{BZAA}, section 2}
\label{BZ}
\begin{sect}
\label{intE}
Let $K$ be a field with $\fchar(K)\ne 2$ and  $\bar{K}$ an algebraic closure of $K$. Let $f(x) \in K[x]$ be a monic cubic polynomial without repeated roots and
$\{\alpha_1,\alpha_2,\alpha_3\} \subset \bar{K}$ the set of roots of $f(x)$. Clearly
$$f(x)=(x-\alpha_1) (x-\alpha_2) (x-\alpha_3) \in \bar{K}[x].$$
Let
$$E=E_f: y^2=f(x)$$
be an elliptic curve over $K$ and $\infty$ its only infinite point (the zero of the group law).
We have
$$\begin{aligned}E(K)=\{(x_0,y_0)\in K^2\mid y_0^2=f(x_0)\}  \coprod\{\infty\}\\
\subset \{(x_0,y_0)\in \bar{K}^2\mid y_0^2=f(x_0)\}\coprod\{\infty\} = E(\bar{K}).\end{aligned}$$
The points
$$W_i=(\alpha_i,0) \in  E(\bar{K}) \ (1 \le i \le 3)$$
are the only points of order 2 on $E$.
\end{sect}

\begin{rem}
\label{onePT}
A root $\alpha_i$ lies in $K$ if and only if $W_i=(\alpha_i,0) \in E(K)$.
This implies that $E(K)$ contains exactly one point of order 2 if and only if $f(x)$ has exactly one root in $K$.
\end{rem}

The following assertion is pretty well-known \cite{Cassels,Knapp} (see also \cite{BZAA}).

\begin{thm}
\label{divCLass}
Suppose that all the roots $\alpha_i$ of $f(x)$ lie in $K$, i.e., all three points of order $2$ on $E$ lie in $E(K)$.
Let $P=(x_0,y_0) \in  E(K)$. Then $P \in 2 E(K)$ if and only if all $x_0-\alpha_i$ are squares in $K$.
In addition, each point $Q\in E(\bar{K})$ with $2Q=P$ actually lies in $E(K)$.
\end{thm}
In what follows we discuss the divisibility by 2 of points in $E(K)$ when not necessarily  all roots of $f(x)$ lie in $K$.

\begin{sect}
\label{bzAA}
Let $P=(x_0,y_0) \in  E(\bar{K})$. There are precisely four points
$Q \in  E(\bar{K})$ such that $2Q=P$. The following explicit construction of all {\sl halves} $Q$'s was described in \cite[Sect. 2]{BZAA}.
Let us choose square roots
\begin{equation}
\label{bzaaR}
r_i=\sqrt{x_0-\alpha_i}\in \bar{K} \  (1\le i \le 3).
\end{equation}
(There are exactly eight choices of such triples $(r_1,r_2,r_3)$.) We have
\begin{equation}
\label{signY}
r_1 r_2 r_3=\sqrt{(x_0-\alpha_1) (x_0-\alpha_2) (x_0-\alpha_3)}=\pm y_0.
\end{equation}
Since all $\alpha_i$'s are {\sl distinct}, $r_i \ne \pm r_j$ if $i \ne j$.  This implies that for all $i \ne j$
\begin{equation}
\label{ineqR}
r_i\pm r_j \ne 0.
\end{equation}
Let us consider only the triples that satisfy
\begin{equation}
\label{bzaaY}
r_1 r_2 r_3=-y_0.
\end{equation}
(There are exactly four choices of such triples $(r_1,r_2,r_3)$.)
Let us put
\begin{equation}
\label{SymR}
s_1=s_1(r_1,r_2,r_3)=r_1+r_2+r_3, \ s_2=s_2(r_1,r_2,r_3)=r_1 r_2 +r_2 r_3+r_3 r_1.
\end{equation}
Then the point
\begin{equation}
\label{bzaaQ}
Q=Q_{r_1,r_2,r_3}=(x_1,y_1):=(x_0+s_2, -y_0-s_1 s_2)=
\end{equation}
$$(x_0+r_1 r_2 +r_2 r_3+r_3 r_1, -(r_1+r_2)(r_2+r_3)(r_3+r_1))$$
lies in  $E(\bar{K})$ and satisfies
$$2Q=Q_{r_1,r_2,r_3}=P \in E(\bar{K}).$$
In addition,
\begin{equation}
\label{bzaaS}
l:=-s_1=-(r_1+r_2+r_3)
\end{equation}
 is the slope of the tangent  line $\mathcal{L}_Q$  to $E$ at $Q$ \cite[Proof of Th. 2.1]{BZAA}. More precisely, $\mathcal{L}_Q$ passes through the point $-P=(x_0,-y_0)$
and is defined by the equation
$$y=l (x-x_0)-y_0.$$

Distinct choices of triples $(r_1,r_2,r_3)$ give rise to distinct halves $Q_{r_1,r_2,r_3}$'s. More precisely \cite[Th. 2.3]{BZAA},
\begin{equation}
\label{plusW}
Q_{r_1,-r_2,-r_3}=Q_{r_1,r_2,r_3}+W_1, Q_{-r_1,r_2,-r_3}=Q_{r_1,r_2,r_3}+W_2, Q_{-r_1,-r_2,r_3}=Q_{r_1,r_2,r_3}+W_3.
\end{equation}
(That is how we get all four $Q\in E(\bar{K})$ with $2Q=P$.)
%%%%%%%%%%%%%%%new stuff added Dec. 1, 2017%%%%%%%%%%
Conversely, if we are given $Q=(x_1,y_1)\in  E(\bar{K})$, then $Q=Q_{r_1,r_2,r_3}$ with
\begin{equation}
\label{QtoR}
r_i=-\frac{y_1}{2}\cdot \left(-\frac{1}{x_1-\alpha_i}+\frac{1}{x_1-\alpha_j}+\frac{1}{x_1-\alpha_k}\right)
\end{equation}
for each permutation $i,j,k$ of $1,2,3$.
%%%%%%%%%%%%%%%%%%%%%%%%%%%%
\end{sect}

Formulas from Sect. \ref{bzAA} almost immediately lead to the following statement.

\begin{thm}
\label{div2S}
Suppose that $P=(x_0,y_0) \in E(K)$.
Suppose that square roots $r_i=\sqrt{x_0-\alpha_i}\in \bar{K}$ ($i=1,2,3$) satisfy
\begin{equation}
\label{rK}
r_1 r_2 r_3=-y_0; \ s_1=r_1+r_2+r_3 \in K, \ s_2=r_1 r_2 +r_2 r_3+r_3 r_1 \in K.
\end{equation}
Then the point $Q_{r_1,r_2,r_3}\in E(\bar{K})$ defined by  formula \eqref{bzaaQ} enjoys the following properties:
$$Q_{r_1,r_2,r_3}\in E(K),  \ 2 Q_{r_1,r_2,r_3}=P.$$
Conversely, suppose that $Q \in E(K)$ satisfies $2Q=P$. Then there exists precisely one triple $\{r_1,r_2,r_3\}$ of square roots $r_i=\sqrt{x_0-\alpha_i}\in \bar{K}$ ($i=1,2,3$)
that satisfy \eqref{rK} and such that $Q=Q_{r_1,r_2,r_3}$.
\end{thm}

\begin{proof}
Suppose that the square roots $r_i$ ($i=1,2,3$) satisfy \eqref{rK}. This implies that both $s_1$ and $s_2$ defined in \eqref{SymR} lie in $K$.
The point
$Q_{r_1,r_2,r_3}=(x_1,y_1)$  defined by  formula \eqref{bzaaQ}  lies in $E(\bar{K})$ and $2Q=P$.
Since, $x_0,y_0, s_1, s_2\in K$,  the formulas \eqref{bzaaQ} imply that $x_1,y_1 \in K$, i.e., $Q\in E(K)$.

Conversely, suppose that $P=2 Q$ with $Q=(x_1,y_1) \in E(K)$. Then
$$Q\in E(K)\subset E(\bar{K}); \ x_1, y_1 \in K.$$
There exists exactly one triple $\{r_1,r_2,r_3\}$ of  square roots $r_i=\sqrt{x_0-\alpha_i} \in \bar{K}$ that satisfy \eqref{bzaaR}  such that $r_1 r_2 r_3=-y_0$ and
$Q=Q_{r_1,r_2,r_3}$. In light of \eqref{bzaaQ},  $s_2=x_1-x_0$. Since both $x_0,x_1\in K$,  we obtain that $s_2 \in K$.
Since $Q \in E(K)$, the slope $l$ of the tangent line to $E$ at $Q$ lies in $K$. In light of \eqref{bzaaS}, $l =-s_1$ and therefore
 $s_1 \in K$. To summarize: $r_1 r_2 r_3=-y_0$ and
 $r_1+r_2+r_3=s_1 \in K$, $\ s_2 \in K.$
 In other words, the triple $\{r_1,r_2,r_3\}$ satisfies the conditions \eqref{rK}.
\end{proof}

\begin{cor}
\label{Crit2}
Suppose that $P=(x_0,y_0) \in E(K)$.
Then the following conditions are equivalent.
\begin{itemize}
\item[(i)]
$P \in 2 E(K)$, i.e., there exists $Q\in E(K)$ such that $2Q=P$.

\item[(ii)]
One may choose square roots $r_i=\sqrt{x_0-\alpha_i}\in \bar{K}$ in such a way that
$$s_1=r_1+r_2+r_3 \in K, \ s_2=r_1 r_2 +r_2 r_3+r_3 r_1 \in K.$$
\item[(iii)]
One may choose square roots $r_i=\sqrt{x_0-\alpha_i}\in \bar{K}$ in such a way that
$$r_1 r_2 r_3=-y_0; \ s_1=r_1+r_2+r_3 \in K, \ s_2=r_1 r_2 +r_2 r_3+r_3 r_1 \in K.$$
\end{itemize}
\end{cor}

\begin{proof}
By Theorem \ref{div2S}, one has only to check that (ii) implies (iii).
Indeed,  suppose we are given a triple $(r_1,r_2,r_3)$ of square roots
$r_i=\sqrt{x_0-\alpha_i}$ such that
$$s_1 =r_1+r_2+r_3 \in K, s_2= r_1 r_2+r_2 r_3+r_3 r_1 \in K.$$
By equality \eqref{signY}, $r_1r_2 r_3=\pm y_0$.
Replacing if necessary $(r_1,r_2,r_3)$
by $(-r_1,-r_2,-r_3)$ (and $s_1$ by $-s_1$), we may and will assume that $r_1 r_2 r_3=-y_0$  and still $s_1, s_2 \in K$.
This proves that (ii) implies (actually equivalent to) (iii).
\end{proof}

\section{Division by 2}
\label{sec2}
In this section we discuss division by 2 in $E(K)$ when $E(K)$ contains exactly one point of order 2.

Let $K$ be a field with $\fchar(K)\ne 2$. Let $\alpha$ be an element of $K$,
$$g(x)=x^2+px+q\in K[x]$$
 a monic irreducible quadratic polynomial (in particular, $\alpha$ is {\sl not} a root of $g(x)$), and
$K_g=K[x]/g(x)K[x]$ the corresponding quadratic field extension of $K$.  We write $\mathbf{x}$ for the image of $x$ in $K_g$ with respect to the natural surjective homomorphism
$$K[x]\twoheadrightarrow K[x]/g(x)K[x]=K_g.$$
Clearly, $\mathbf{x}^2+p \mathbf{x}+q=0.$
Let $\iota: K_g \to K_g$ be the only nontrivial $K$-linear  automorphism (involution) of  $K_g$. We write
$$\mathrm{Tr}:  K_g \to K, \ z \mapsto z+\iota(z),  \ \mathrm{Norm}:  K_g \to K, \ z \mapsto z\cdot\iota(z)$$
for the trace and norm maps attached to $K_g/K$. We have
$$z^2-\Tr(z)+\Norm(z)=(z-z)(z-\iota(z))=0 \ \forall z\in K_g.$$
For each $a \in K$ we have
$$g(a)=\Norm(a-\mathbf{x}).$$
We also have
$$\iota(\mathbf{x})= -p-\mathbf{x}, \ \Norm(\mathbf{x})=q, \ \Tr(\mathbf{x})=-p.$$
In addition, we have the following equality in $K_g[x]$:
$$g(x)=x^2+px+q=(x-\mathbf{x})(x-\iota \mathbf{x})=(x-\mathbf{x})(x+p+\mathbf{x}).$$
Let us consider the elliptic curve
\begin{equation}
\label{equationE}
E=\mathcal{E}_{\alpha,p,q}: y^2=f(x):=(x-\alpha) g(x).
\end{equation}
Clearly, $E(K)$ contains exactly one point of order 2, namely, $(\alpha,0)$.
On the other hand, $E(K_g)$ contains two remaining points of order $2$,
namely, $(\mathbf{x},0)$ and $(\iota\mathbf{x},0)=(-p-\mathbf{x},0)$.

We start with the following result that was inspired by a paper of Schaefer \cite{Schaefer}.

\begin{thm}
\label{division}
A point $P=(x_0,y_0)\in E(K)$ lies in $2 E(K)$ if and only if
$x_0-\mathbf{x}$ is a square in $K_g$.
If this is the case, then $x_0-\alpha$ is a square in $K$.
\end{thm}

\begin{rem}
\label{squareR}
\begin{itemize}
\item[(i)]
If $x_0 \in K$, then $x_0-\mathbf{x} \ne 0$, because $\mathbf{x} \not\in K$.
\item[(ii)]
Suppose that $\rr \in K_g$  and $r\in K$ satisfy
$$r^2=x_0-\alpha, \ \rr^2=x_0-\mathbf{x}.$$
Then
$$\left(r \cdot \Norm(\rr)\right)^2=r^2 \cdot \Norm(x_0-\mathbf{x})=(x_0-\alpha)g(x_0)=y_0^2.$$
This means that
\begin{equation}
\label{sign}
r \cdot \Norm(\rr)=\pm y_0.
\end{equation}
In addition,
$$(\Tr({\rr}))^2=(\rr+\iota(\rr))^2=\rr^2+2\rr\iota(\rr)+\iota(\rr)^2=
\left(\rr^2+\iota(\rr^2)\right)+2\Norm(\rr)$$
$$=\Tr(x_0-\mathbf{x})+2\Norm(\rr)=2x_0+p+2\Norm(\rr).$$
It follows that
\begin{equation}
\label{sign2}
(\Tr({\rr}))^2=2x_0+p+2\Norm(\rr).
\end{equation}
Notice also that
$$\begin{aligned}\Norm(r+\rr)=(r+\rr)\iota(r+\rr) =(r+\rr)(r+\iota(\rr))\\=r^2+r(\rr+\iota(\rr))+\rr \iota(\rr)=(x_0-\alpha)+r\Tr(\rr)+\Norm(\rr).\end{aligned}$$
Consequently,
\begin{equation}
\label{sign3}
\Norm(r+\rr)=(x_0-\alpha)+r\Tr(\rr)+\Norm(\rr).
\end{equation}
\end{itemize}
\end{rem}

The next assertion describes explicitly both $\frac{1}{2} P$'s in $E(K)$.

\begin{thm}
\label{divExp}
Suppose that $P=(x_0,y_0)\in E(K)$.  Suppose that $r\in K$ and $\rr \in K_g$ satisfy
$$r^2=x_0-\alpha, \ \rr^2=x_0-\mathbf{x}, \ r \Norm(\rr)=-y_0.$$
Then:
\begin{itemize}
\item[(i)]
 the points
$$Q_{r,\rr}:=Q_{\rr,\iota(\rr),r}(x_1,y_1):= \left(x_0+\Tr(\rr)r+\Norm(\rr), -y_0-\left(r+\Tr(\rr)\right)\left(\Tr(\rr)r+\Norm(\rr)\right)\right)=$$
$$\left(\alpha+\Norm(r+\rr), -\Tr(\rr)\Norm(r+\rr)\right)$$
and
$$\begin{aligned}Q_{r,-\rr}:=Q_{-\rr,-\iota(\rr),r}=(x_{-1},y_{-1}):=\\ \left(x_0-\Tr(\rr)r+\Norm(\rr), -y_0-\left(r-\Tr(\rr)\right)\left(-\Tr(\rr)r+\Norm(\rr)\right)\right)\\ =
 \left(\alpha+\Norm(r-\rr), \Tr(\rr)\Norm(r-\rr)\right)\end{aligned}$$
are distinct points of $E(K)$ that satisfy
$$2 Q_{r,\rr}=2 Q_{r,-\rr}=P \in E(K).$$
\item[(ii)]
The tangent lines to $E$ at $Q_{r,\rr}$ and $Q_{r,-\rr}$ are defined by the equations
$$\mathcal L_{r,\rr}: y=l_{r,\rr}(x-x_0)-y_0 \ \text{ and }\mathcal L_{r,-\rr}: y=l_{r,-\rr}(x-x_0)-y_0,$$
respectively, where
$$l_{r,\rr}=-\left(r+\Tr(\rr)\right), \ l_{r,-\rr}=-\left(r-\Tr(\rr)\right).$$

\end{itemize}
\end{thm}

\begin{rem}
\label{choicer}
Suppose we are given $\rr \in K_g$ with $\rr^2=x_0-\mathbf{x}$. It follows from Theorem \ref{division} and \eqref{sign} that there exists precisely one $r \in K$ such that $r^2=x_0-\alpha$
and $r\Norm(\rr)=-y_0$.
\end{rem}

\begin{proof}[Proof of Theorem \ref{division} and \ref{divExp}]

We start with the proof of Theorem \ref{division}.
Suppose that $P=(x_0,y_0) \in 2 E(K)$.
Since $E(K_g)$ contains all three points of order 2 on $E$ and $(\mathbf{x},0)$ is one of them; it follows from Theorem \ref{divCLass} that
$x_0-\mathbf{x}$ is a  square in $K_g$.

Conversely, suppose that $x_0-\mathbf{x}$ is a  square in $K_g$. This means that
there exists  $\rr \in K_g$
such that  $$x_0-\mathbf{x}=\rr^2=(-\rr)^2.$$
 By Remark \ref{squareR}(i),  $\rr$ does not lie in $K$. In particular,  $\rr\ne 0$.
We have
$$x_0-\iota(\mathbf{x})=\iota(x_0-\mathbf{x})=\iota(\rr^2)=\iota(\rr)^2.$$
This implies that
$$g(x_0)=(x_0-\mathbf{x})(x_0-\iota\mathbf{x})=(\rr\iota(\rr))^2=\Norm(\rr)^2$$
is a square in $K$, because $\Norm(\rr)\in K$. Since $(x_0,y_0) \in E(K)$,
$$y_0^2=f(x_0)=(x_0-\alpha) g(x_0)=(x_0-\alpha) \Norm(\rr)^2$$
is a square in $K$.  This implies that $x_0-\alpha$ is also a square in $K$, because the norm of nonzero $\rr$ also does {\sl not} vanish.
It follows that $x_0-\alpha$ is  a square in $K_g$ as well.
By Remark \ref{squareR}(ii), we may
choose
$$r =\sqrt{x_0-\alpha}\in K\subset K_g$$ in such a way that
$$r \Norm(\rr)=r \Norm(-\rr)= r \cdot\rr\cdot\iota(\rr)=-y_0.$$
Since all three  $x_0-\mathbf{x}$, $x_0-\iota(\mathbf{x})$, and $x_0-\alpha$ are squares in $K_g$, Theorem \ref{divCLass} implies that
$$P=(x_0,y_0) \in 2 E(K_g);$$
by formula \eqref{bzaaQ} combined with the condition $r\Norm(\rr)=-y_0$ and equality \eqref{sign3},  the two different choices   $(\rr, \iota(\rr),r)$ and  $(-\rr, -\iota(\rr),r)$ of the corresponding square roots give us
$$Q_{r,\rr}=\left(x_0+(\rr+\iota(\rr))r+\rr\iota(\rr)), -y_0-\left(r+\rr+\iota(\rr)\right)((\rr+\iota(\rr))r+\rr\iota(\rr) )\right)$$
 $$=\left(x_0+\Tr(\rr)r+\Norm(\rr), -y_0-\left(r+\Tr(\rr)\right)\left(\Tr(\rr)r+\Norm(\rr\right)\right))$$
$$=\left(\alpha+\Norm(r+\rr), -\Tr(\rr)\Norm(r+\rr)\right)$$
and
$$Q_{r,-\rr}= \left(x_0-(\rr+\iota(\rr))r+\rr\iota(\rr), -y_0-\left(r-\rr-\iota(\rr)\right)\left((-\rr-\iota(\rr))r+\rr\iota(\rr)\right)\right)=$$
$$\left(\alpha+\Norm(-r+\rr), \Tr(\rr)\Norm(r-\rr)\right)=$$
$$\left(x_0-\Tr(\rr)r+\Norm(\rr), -y_0-\left(r-\Tr(\rr)\right)\left(-\Tr(\rr)r+\Norm(\rr)\right)\right)$$
that are distinct points of $E(K_g)$ such that
$$2 Q_{r,\rr}=2 Q_{r,-\rr}=P \in E(K).$$
However, it is clear that they both lie in $E(K)$, since all three $r, \Tr(\rr)$, and $\Norm(\rr)$ lie in $K$. This
ends the proof of Theorem \ref{division}. On the other hand,
$Q_{r,\rr}$ and $Q_{r,-\rr}$ are exactly the points that appear in the statement of Theorem \ref{divExp}; this proves Theorem \ref{divExp}(i) as well.
Theorem \ref{divExp}(ii)  follows from the explicit formula \eqref{bzaaS} for the slope $l$ of the tangent line at $\frac{1}{2}P$.
%\cite[Proof of Th. 2.1]{BZAA}.
\end{proof}

\begin{rem}
If $P=(x_0,y_0)\in 2 E(K)$, then $g(x_0)$ is a square in $K$. Indeed, we have
$$y_0^2=f(x_0)=g(x_0)(x_0-\alpha),  \ g(x_0)=\Norm(x_0-\mathbf{x}).$$
%Indeed, $g(x_0)=\Norm(x_0-\tilde{x})$.
 Since $x_0-\mathbf{x}$ is a square in $K_g$, its norm $g(x_0)$ is a square in $K$.
%(Here is another proof that works for ``almost all'' $P$. By Theorem \ref{division}, $x_0-\alpha$ is a square in $K$. So, if $x_0-\alpha\ne 0$  (i.e., $x_0\ne \alpha$) then $g(x_0)$ is a ratio of two squares in $K$ and therefore is also a square in $K$.)
\end{rem}

\begin{ex}
\label{ex4}
Suppose  $g(x)=x^2+px+q \in K[x]$ is an irreducible quadratic polynomial over $K$
and $-\mathbf{x}$ is a square in $K_g$, i.e., there are exist $u,v \in K$ such that
$$-\mathbf{x} =(u\mathbf{x}+v)^2$$
in $K_g$. This means that the polynomial $(ux+v)^2 +x$ is divisible by $x^2+px+q$ in $K[x]$, i.e.,
%This means that
%$$(ux+v)^2 +x-u^2(x^2+px+q)$$
%is divisible by  $x^2+px+q$ in $K[x]$. The latter means that
 the polynomial
$$(2uv+1-u^2 p)x+(v^2-u^2q)=(ux+v)^2 +x-u^2(x^2+px+q)$$ of degree $<2$
is divisible by quadratic  $x^2+px+q$ in $K[x]$ and therefore is identically zero, i.e.,
$$2uv+1=u^2 \cdot p, \ v^2=u^2 q.$$
In addition, $u \ne 0$, since $-\mathbf{x} \in K_g \setminus K$.
This implies that there exists $b \in K$ such that
$$q=b^2, \ v=ub.$$
It follows that
$$u^2 \cdot p=2 u^2 b+1, \ (p-2b) u^2=1.$$
This means that
$$p=2b+\frac{1}{u^2}, q=b^2, \ v=ub.$$
Let us put
$$a:=\frac{1}{u} \in K \setminus \{0\}.$$
Then we obtain that $$p=a^2+2b, \ q=b^2,$$
$$g(x)=x^2+(a^2+2b)x+b^2; \  u=\frac{1}{a},  \ v =\frac{b}{a},$$
and
$$\left(\frac{1}{a}\mathbf{x}+\frac{b}{a}\right)^2=-\mathbf{x}$$
in $K_g$.
Since $g(x)$ is irreducible over $K$, its discriminant
$$(a^2+2b)^2-4b^2=a^4+4a^2b=a^2(a^2+4b)$$
is {\sl not} a square in $K$, i.e, $a^2+4b$ is {\sl not} a square in $K$.
In addition, the constant term $b^2$ of $g(x)$ does {\sl not} vanish, i.e.,
$b \ne 0$.
Now let us consider the $K$-point $W=(0,0)$ of the elliptic curve
$$\mathcal{E}^{(4)}_{a,b}=\mathcal{E}_{0, a^2+2b,b^2}: y^2= \left(x^2+(a^2+2b)x+b^2\right)x.$$
Clearly, $W$ has order 2. It follows from Theorem \ref{division} that $W$ is divisible by 2 in $E(K)$,
i.e., there are two points $Q_{+}$ and $Q_{-}$ of order 4 in $\mathcal{E}^{(4)}_{a,b}(K)$ such that
$$2 Q_{+}=2Q_{-}=W.$$
In order to find explicitly $Q_{+}$ and $Q_{-}$, let us apply Theorem \ref{divExp}. We have
$$r=0, \ \rr=\frac{1}{a}(\mathbf{x}+b);$$
$$\Tr(\rr)=\frac{1}{a} (-p)+2 \cdot \frac{b}{a}=- \frac{(a^2+2b)}{a}+\frac{2b}{a}=-a,$$
$$\Norm(\rr)=\frac{\Norm(\mathbf{x}+b)}{a^2}=\frac{g(-b)}{a^2}=-b.$$
%$$\frac{b^2-2b(a^2+2b)+b^2}{a^2}=-2b+\frac{-2b^2}{a^2}.$$
Now the explicit formulas of Theorem \ref{divExp} give us
$$Q_{+}=(-b,ab), \ Q_{+}=(-b,-ab).$$
\end{ex}

\section{Another criterion of divisibility by 2}
\label{reducible}
We keep the notation of Section \ref{sec2}. However,
 we drop the assumption that $g(x)=x^2+px+q\in K[x]$ is irreducible over $K$ and assume only that it has {\sl no} multiple roots.  Let
%$\bar{K}$ be an algebraic closure of $K$ and
$\alpha_1, \alpha_2$ be (distinct) roots of $g(x)$ in $\bar{K}$. We have
\begin{equation}
\label{Vieta}
\alpha_1+\alpha_2= -p, \ \alpha_1 \alpha_2=q.
\end{equation}

\begin{thm}
\label{divRed}
Let $\alpha$ be an element of $K$ that is not a root of $g(x)$. Let us consider the elliptic curve
$$E: y^2=(x-\alpha)g(x)$$
over $K$. Let
 $$W_3=(\alpha_3,0)=(\alpha,0) \in E(K)$$
 be a point of order $2$.

Let $P=(x_0,y_0)$ be a $K$-point of $E$. Then $P$ is divisible by $2$ in $E(K)$ if and only if the following two conditions hold.
\begin{itemize}
\item[(i)]
The difference $x_0-\alpha$ is a square in $K$.
\item[(ii)]
There exist square roots $r=\sqrt{x_0-\alpha}\in K$ and
$$r_i =\sqrt{x_0-\alpha_i} \in \bar{K} \ (i=1,2)$$
such that
\begin{equation}
\label{prodNot2}
r_1 r_2 r=-y_0;   \ r_1+r_2 \in K.
%, \ r_1 r_2 \in K.
\end{equation}
\end{itemize}
In addition, if  square roots $r_1,r_2,r_3$ satisfy conditions (ii),  then
$$Q_{+}=Q_{r_1,r_2,r}=(x_0+(r_1+r_2)r+r_1 r_2,  -y_0-((r_1+r_2)+r)((r_1+r_2)r+r_1 r_2)$$
and
$$Q_{-}=Q_{-r_1,-r_2,r}=(x_0-(r_1+r_2)r+r_1 r_2,  -y_0-(-(r_1+r_2)+r)(-(r_1+r_2)r+r_1 r_2)$$
are distinct points of  $E(K)$, which satisfy
 $$2Q_{+}=2 Q_{-}=P,  \ Q_{-}=Q_{+}+W_3.$$
 Conversely, if $Q\in E(K)$ satisfies $2Q=P$, then there exist square roots $r,r_1,r_2$ that satisfy conditions (ii)
 and $Q=Q_{r_1,r_2,r}$.
\end{thm}

\begin{rem}
\label{not22}
If $x_0\ne \alpha$, i.e., $P\ne (\alpha,0)$,  then $r=\sqrt{x_0-\alpha} \ne 0$ and
$$r_1 r_2 =-\frac{y_0}{r}.$$
It follows from Theorem \ref{divRed} that   $$P=(x_0,y_0)\in E(K)$$ is divisible by 2 in $E(K)$ if and only if there exist square roots
$$r=\sqrt{x_0-\alpha} \in K, \
r_i =\sqrt{x_0-\alpha_i} \in \bar{K} \ (i=1,2)$$
such that
\begin{equation}
\label{Tr2}
r_1 r_2 =-\frac{y_0}{r}, \   \ r_1+r_2 \in K.
\end{equation}
\end{rem}

\begin{proof}[Proof of Theorem \ref{divRed}]
Let us put
$$\alpha_3=\alpha, \ W_3=(\alpha_3,0)\in E(K).$$
Suppose that conditions (i)-(ii) hold. Let us put
$$r_3=r \in K.$$
We have
$$\begin{aligned}(r_1+r_2)^2=r_1^2+2 r_1 r_2+r_2^2=(x_0-\alpha_1)+2 r_1 r_2+(x_0-\alpha_2)\\=2x_0-(\alpha_1+\alpha_2)+2r_1 r_2=(2 x_0+p)+2 r_1 r_2.\end{aligned}$$
Since all three $r_1+r_2, x_0,p$ lie in $K$, we obtain that
$$r_1 r_2 \in K.$$
Notice that
\begin{equation}
\label{symmetric}
s_1=r_1+r_2+r_3=(r_1+r_2)+r, \ s_2=r_1 r_2+r_2 r_3+r_3 r_1=(r_1+r_2)r+r_1 r_2.
\end{equation}
Clearly, both $s_1$ and $s_2$ lie in $K$. Let us put
\begin{equation}
\label{explicitS}
x_1=x_0+r_1 r_2+r_2 r_3+r_3 r_1=x_0+s_2, \end{equation}
$$y_1=-y_0-(r_1+r_2+r_3)(r_1 r_2+r_2 r_3+r_3 r_1)=-y_0-s_1 s_2.$$
Since all $x_0,y_0, s_1, s_2$ lie in $K$, both $x_1$ and $y_1$ also lie in $K$.
By formula \eqref{bzaaQ}, the point $Q_{+}=Q_{r_1,r_2,r_3}=(x_1,y_1)$ lies in $E(\bar{K})$ and satisfies $2Q_{+}=P$.
Since $x_1$ and $y_1$ lie in $K$, the point $Q_{+}$ actually lies in $E(K)$, and therefore $P$ is divisible by 2 in $E(K)$.
Similarly, the triple of square roots $(-r_1,-r_2,r_3)$  satisfies \eqref{bzaaR} and \eqref{bzaaY}; in addition,
$$(-r_1)+(-r_2)=-(r_1+r_2)\in K, (-r_1)(-r_2)-r_1 r_2.$$
This implies that
$$Q_{-}=Q_{-r_1,-r_2,r_3}=Q_{r_1,r_2,r_3}+W_3=Q_{+}+W_3$$
 also lies in $E(K)$ and satisfies $2Q_{-}=P$.

Conversely, suppose that $P=2Q$ with $Q=(x_1,y_1) \in E(K)$. We claim that
$x_0-\alpha_3=x_0-\alpha$ is a square in $K$. Indeed, if $g(x)$ is irreducible over $K$, then
our claim follows from Theorem \ref{division}. If $g(x)$ is reducible, i.e., $\alpha_1, \alpha_2$ and $\alpha_3$
lie in $K$, then our claim follows from Theorem \ref{divCLass}.

% Let $Q=(x_1,y_1)$ be a $K$-point of $E$ such that $2Q=P$.
%According to Sect. \ref{bzAA} and \eqref{bzaaQ}, there exist square roots
%$r_i=\sqrt{x_0-\alpha_i}$  such   that
%$$Q=Q_{r_1,r_2,r_3}=(x_0+s_2, -y_0-s_1 s_2).$$
It follows from Theorem \ref{div2S} that there exist square roots
$r_i=\sqrt{x_0-\alpha_i}$  such that
$$r_1 r_2 r_3=-y_0, r_1+r_2+r_3  \in K, r_1 r_2+r_2 r_1+r_3 r_1 \in K$$
and
$$Q=(x_1,y_1)=Q_{r_1,r_2,r_3}.$$
Notice that
$$r:=r_3=\sqrt{x_0-\alpha_3}=\sqrt{x_0-\alpha}\in K$$
and therefore
$$r_1 r_2 r=r_1 r_2 r_3=-y_0, \ Q=Q_{r_1,r_2,r_3}=Q_{r_1,r_2,r}.$$
Since $r_3=r \in K$, we conclude that
$r_1+r_2=(r_1+r_2+r_3)-r_3\in K$.
%, \ r_1 r_2=(r_1 r_2+r_2 r_1+r_3 r_1)-(r_1+r_2)r_3 \in K.$$
\end{proof}
%%%%%%%%new stuff added December 1, 2017%%%%%%
\begin{rem}
\label{doubleQ}
Let $Q =(x_1,y_1) \in E(K)$ with $y_1 \ne 0$ (i.e., $Q$ is none of $W_j$) and $P=2Q \in E(K)$.  Then $Q=Q_{r_1,r_2,r_3}$
for suitable $r_k=\sqrt{x(P)-\alpha_k}$,
and formulas \eqref{QtoR} (with $i=3$) tell us that
$$\begin{aligned}r_3=-\frac{y_1}{2}\cdot \left(-\frac{1}{x_1-\alpha}+\frac{1}{x_1-\alpha_1}+\frac{1}{x_1-\alpha_2}\right)\\=
 -\frac{y_1}{2}\cdot \left(-\frac{1}{x_1-\alpha}+\frac{2x_1+p}{g(x_1)}\right).\end{aligned}$$
In particular, if $x_1=0$ and $\alpha=-1$, then $y_1^2=q$, $g(x_1)=q$  and
$$r_3^2= \frac{q}{4}\cdot  \left(-1+\frac{p}{q}\right)^2=  \frac{(p-q)^2}{4q}$$
and
\begin{equation}
\label{doubleX}
x(P)=x(2Q)=\alpha+r_3^2= \frac{(p-q)^2}{4q}-1.
\end{equation}
\end{rem}
%%%%%%%%%%%%%%%%%%%%%%%%%%%%%
\begin{thm}
\label{divideS}
We keep the notation and assumptions of Theorem \ref{divRed}.
Assume additionally that $x_0 \ne \alpha$, i.e., $P \ne (\alpha,0)$.
Then $P=(x_0,y_0)$ is divisible by $2$ in $E(K)$ if and only if there exists $r \in K$ such that $r^2=x_0-\alpha$ and
$(2x_0+p)(x_0-\alpha)-2y_0 r$
is a  square in $K$.

In addition, if these equivalent conditions hold and we choose
$$T=\frac{\sqrt{(2x_0+p)(x_0-\alpha)-2y_0 r}}{r}\in K,$$
then
$$\mathcal{Q}_{r,T}=(x_1,y_1):= \left(x_0+rT-\frac{y_0}{r}, -y_0-\left(r+T\right)\left(rT +\frac{y_0}{r}\right)\right)$$
and
$$\mathcal{Q}_{r,-T}=(x_{-1},y_{-1}):=\left(x_0-rT-\frac{y_0}{r}, -y_0-\left(r-T\right)\left(-rT  +\frac{y_0}{r}\right)\right)$$
are distinct points of $E(K)$, which satisfy
$$2 \mathcal{Q}_{r,T}=2 \mathcal{Q}_{r,-T}=P \in E(K), \   \mathcal{Q}_{r,-T}= \mathcal{Q}_{r,T}+W_3.$$
Conversely, if $Q\in E(Q)$ satisfies $2Q=P$, then there exist
$$r=\sqrt{x_0-\alpha}\in K \ \text{ and }\ T=\frac{\sqrt{(2x_0+p)(x_0-\alpha)-2y_0 r}}{r}\in K$$
such that $Q=\mathcal{Q}_{r,T}$.
\end{thm}

\begin{proof}
In light of Remark \ref{not22}, the divisibility of $P$ by 2 in $E(K)$ is equivalent to the existence of square roots
$r=\sqrt{x_0-\alpha}\in K$ and $r_i =\sqrt{x_0-\alpha_i}\in \bar{K}$ ($i=1,2$) such that
$$r_1 r_2 =-\frac{y_0}{r},   \ r_1+r_2 \in K.$$
Suppose such $ r_1, r_2, r$ exist. Then $(r_1+r_2)^2$ is a square in $K$. We have
\begin{equation}
\label{r1PLUSr2S}
\begin{aligned}
(r_1+r_2)^2=r_1^2+2 r_1 r_2+ r_2^2=(x_0-\alpha_1)-2\frac{y_0}{r}+(x_0-\alpha_2)\\=
2x_0+p-2\frac{y_0}{r}=\frac{(2x_0+p)(x_0-\alpha)-2y_0 r}{r^2}.\end{aligned}\end{equation}
This implies that $(2x_0+p)(x_0-\alpha)-2y_0 r$ is a {\sl square} in $K$.

Conversely, suppose that there exists $r\in K$ such that $r^2=x_0-\alpha$ and $(2x_0+p)(x_0-\alpha)-2y_0 r$ is a {\sl square} in $K$.
This implies that
$$\frac{(2x_0+p)(x_0-\alpha)-2y_0 r}{r^2}=\frac{(2x_0+p)(x_0-\alpha)-2y_0 r}{x_0-\alpha}=2x_0+p-\frac{2 y_0}{r}$$
is a square in $K$. Let us put $r_3=r$ and choose square roots $r_i=\sqrt{x_0-\alpha_i} \in \bar{K}$ ($i=1,2$) in such a way that
$r_1 r_2 r_3=-y_0$.  (Since $r_3=r \ne 0$, the only other choice of such a pair of square roots is $(-r_1,-r_2)$.)  Let
$$T:=r_1+r_2.$$
We have
$$r_1 r_2=-\frac{y_0}{r}\in K$$ and
$$T^2=(r_1+r_2)^2=r_1^2+2 r_1 r_2+ r_2^2=(x_0-\alpha_1)-2\frac{y_0}{r}+(x_0-\alpha_2)=$$
$$2x_0+p-2\frac{y_0}{r}=\frac{(2x_0+p)(x_0-\alpha)-2y_0 r}{r^2}$$
is a {\sl square} in $K$. This implies that
$$T=r_1+r_2 \in K.$$
It follows from Remark \ref{not22} that $P$ is divisible by 2 in $E(K)$. By Theorem \ref{divRed} we obtain that
$$Q_{+}=Q_{r_1,r_2,r}=(x_0+(r_1+r_2)r+r_1 r_2, -y_0-((r_1+r_2)+r)((r_1+r_2)r+r_1 r_2)$$
and
$$Q_{-}=Q_{-r_1,-r_2,r}=(x_0-(r_1+r_2)r+r_1 r_2, -y_0-(-(r_1+r_2)+r)(-(r_1+r_2)r+r_1 r_2)$$
are distinct points of  $E(K)$, which satisfy
 $$2Q_{+}=2 Q_{-}=P,  \ Q_{-}=Q_{+}+W_3.$$
 Taking into account that
 $$r_1+r_2=T, \ r_1 r_2=-\frac{y_0}{r},$$
 we obtain that
 $$\mathcal{Q}_{r,T}=Q_{+}, \ \mathcal{Q}_{r,-T}=Q_{-},$$
 and therefore
 $$\mathcal{Q}_{r,T}, \mathcal{Q}_{r,-T}\in E(K), \ \mathcal{Q}_{r,-T}=\mathcal{Q}_{r,T}+W_3,$$
 $$2 \mathcal{Q}_{r,T}=2 \mathcal{Q}_{r,-T}=P.$$
On the other hand, if $P=2Q$ with $Q\in E(K)$, then it follows from Theorem \ref{divRed} that there exist square roots $r=\sqrt{x_0-\alpha}\in K$ and
$$r_i =\sqrt{x_0-\alpha_i} \in \bar{K} \ (i=1,2)$$
such that
$$r_1 r_2 r=-y_0,   \ r_1+r_2 \in K$$ and $Q=Q_{r_1,r_2,r}$.  This implies that $r_1 r_2=-y_0/r$. Now if we put $T=(r_1+r_2)\in K$, then (as we have already seen)
$$T^2=(r_1+r_2)^2=(2x_0+p)(x_0-\alpha)-2y_0 r\in K.$$
%In particular, $(2x_0+p)(x_0-\alpha)-2y_0 r$ is a square in $K$.
 Now one can easily check that
$$Q_{r_1,r_2,r}=(x_0+(r_1+r_2)r+r_1 r_2, -y_0-((r_1+r_2)+r)((r_1+r_2)r+r_1 r_2)=\mathcal{Q}_{r,T},$$
and therefore $Q=\mathcal{Q}_{r,T}$.

\end{proof}

\begin{rem}
\label{nonvan}
It is known \eqref{ineqR} that  $r_1+r_2 \ne 0$. In light of \eqref{r1PLUSr2S},
$$(2x_0+p)(x_0-\alpha)-2y_0 r \ne 0.$$
\end{rem}

\begin{thm}
\label{divRed2}
Let $\alpha$ be an element of $K$ that is not a root of $g(x)$. Let us consider the elliptic curve
$$E: y^2=g(x)(x-\alpha)$$
over $K$. Let $y_0$ be a nonzero element of $K$.
%Let $P=(x_0,y_0)$ be a $K$-point of $E$.
Then
$$P=(0,y_0) \in 2 E(K)\subset E(K)$$
if and only if there exists $r \in K$ and a nonzero $T \in K$ such that %either
\begin{equation}
\label{rSquareP}
r^2=-\alpha, \ g(x)=x^2+\left(T^2+2\frac{y_0}{r}\right)x+\left(\frac{y_0}{r}\right)^2.
\end{equation}
%or
%\begin{equation}
%\label{rSquareM}
%g(x)=x^2+\left(T^2-2\frac{y_0}{r}\right)+\left(\frac{y_0}{r}\right)^2.
%\end{equation}
If  \eqref{rSquareP} holds, then the equation for $E$ becomes
\begin{equation}
\label{ellTr}
y^2=(x+r^2)\left(x^2+\left(T^2+2\frac{y_0}{r}\right)x+\left(\frac{y_0}{r}\right)^2\right).
\end{equation}
In addition,
$$\mathcal{Q}_{r,T}=\left(r T-\frac{y_0}{r}, -y_0-(r+T)\left(r T-\frac{y_0}{r}\right)\right)$$
  and
 $$\mathcal{Q}_{r,-T}=\left(-r T-\frac{y_0}{r}, -y_0-(r-T)\left(-r T-\frac{y_0}{r}\right)\right)$$
are distinct points of  $E(K)$ that satisfy
$$2 \mathcal{Q}_{r,T}=2 \mathcal{Q}_{r,-T}=P,   \  \mathcal{Q}_{r,-T}=\mathcal{Q}_{r,T}+W_3 .$$
%If  \eqref{rSquareM} holds then the points
%$$Q_{-r,T}=\left(-r T+\frac{y_0}{r}, -y_0-(-r+T)\left(-r T+\frac{y_0}{r}\right)\right) \ \text{ and } Q_{-r,-T}=\left(r T+\frac{y_0}{r}, -y_0+(r+T)\left(r T+\frac{y_0}{r}\right)\right)$$
%lie in $E(K)$ and satisfy
%$$2 Q_{-r,T}=2 Q_{-r,-T}=P.$$
Moreover, if $Q\in E(K)$ satisfies $2Q=P$, then there exist
$$r\in K \ \text{ and }\ T\in K$$
that enjoy   property  \eqref{rSquareP} and
 $Q=\mathcal{Q}_{r,T}$.
\end{thm}

\begin{proof}
We have the equation for $E$
$$y^2=(x^2+px+q)(x-\alpha).$$
%with $\alpha=-r^2$.
Since $P=(x_0,y_0)=(0,y_0) \in E(K)$, we have
$ y_0^2=q\cdot (-\alpha),$
which means that
$$q=\frac{y_0^2}{-\alpha}.$$
In light of Theorem \ref{divideS}, the inclusion $P\in 2 E(K)$ is equivalent to the existence of nonzero $r,T \in K$ such that
$$r^2=-\alpha, \ T^2=\frac{(2x_0+p)(x_0-\alpha)-2y_0 r}{r^2}=\frac{p r^2-2 y_0 r}{r^2}=p- 2 \frac{y_0}{r}.$$
This means that
$$q=\left(\frac{y_0}{r}\right)^2, \    p=T^2+2 \frac{y_0}{r},$$
i.e.,
$$g(x)=x^2+px+q=x^2+\left(T^2+2 \frac{y_0}{r}\right)x+\left(\frac{y_0}{r}\right)^2.$$

The  assertions about $\mathcal{Q}_{r,T}$,  $\mathcal{Q}_{r,-T}$, and $Q=\frac{1}{2}P$ follow from the corresponding assertions of Theorem \ref{divideS}.
\end{proof}

\begin{rem}
\label{req1}
Dividing both sides of \eqref{ellTr} by $r^6$ and making the substitution
$$\tilde{x}=\frac{x}{r^2}, \tilde{y}=\frac{y}{r^3}, t=\frac{T}{r}, \tilde{y}_0=\frac{y_0}{r^3},$$
we obtain that $E$ is isomorphic to the elliptic curve
\begin{equation}
\label{ellt1}
\tilde{E}_{t,\tilde{y}_0}:\tilde{y}^2=(\tilde{x}+1)\left(\tilde{x}^2+\left(t^2+2\tilde{y}_0\right)\tilde{x}+\tilde{y}_0^2\right).
\end{equation}
In addition, the isomorphism
$$E \cong \tilde{E}_{t,\tilde{y}_0}: \ (x,y) \mapsto (\tilde{x},\tilde{y})=\left(\frac{x}{r^2}, \frac{y}{r^3}\right)$$
sends $W_3=(-r^2,0)\in E(K)$ to $\tilde{W}_3=(-1,0) \in \tilde{E}_{t,\tilde{y}_0}(K)$,  $P$ to $\tilde{P}=(0, \tilde{y}_0)\in \tilde{E}_{t,\tilde{y}_0}(K)$,
$\mathcal{Q}_{r,T}$ to $\mathcal{Q}_{1,t}=\left(t-\tilde{y}_0, -\tilde{y}_0-(1+t)\left(t-\tilde{y}_0\right)\right)$, and
$\mathcal{Q}_{r,-T}$ to $\mathcal{Q}_{1,-t}=\left(-t-\tilde{y}_0, -\tilde{y}_0-(1-t)\left(-t-\tilde{y}_0\right)\right)$.
We also have
$$2 \mathcal{Q}_{1,t}=2 \mathcal{Q}_{1,-t}=\tilde{P}, \ \mathcal{Q}_{1,-t}=\mathcal{Q}_{1,t}+\tilde{W}_3.$$
\end{rem}

\section{Elliptic curves with points of order 4}

Let $E$ be an elliptic curve over a field $K$ with $\fchar(K)\ne 2$. Suppose that $E(K)$ contains exactly one point of order 2.
Then $E$ is $K$-isomorphic to
$$\mathcal{E}_{0,p,q}: y^2=x(x^2+px+q),$$
where $x^2+px+q\in K[x]$ is an irreducible quadratic polynomial.

\begin{thm}
\label{order44}
The following conditions are equivalent.
\begin{itemize}
\item[(i)]
$E(K)$ contains a cyclic subgroup of order $4$.
\item[(ii)]
There exist nonzero $a,b \in K$ such that $a^2+4b$ is not a square in $K$ and $E$ is $K$-isomorphic to the elliptic curve
$$\mathcal{E}^{(4)}_{a,b}: y^2=x\left(x^2+(a^2+2b)x+b^2\right).$$
\end{itemize}
\end{thm}

\begin{proof}
Suppose (i) holds. We may assume that $E=\mathcal{E}_{0,p,q}$.
It follows from Example \ref{ex4} that there are nonzero $a,b \in K$ such that $a^2+4b$ is not a square in $K$,
$p=a^2+2b$, and $q=b^2.$
This means that $E=\mathcal{E}^{(4)}_{a,b}$, i.e., (ii) holds.

Suppose (ii) holds. We may assume that
$$E=\mathcal{E}^{(4)}_{a,b}: y^2=x\left(x^2+(a^2+2b)x+b^2\right)$$
with nonzero
$a,b \in K$
and   $a^2+4b$ is  {\sl not} a square in $K$. (In particular, $a^2+4b \ne 0$.)
Clearly, $W=(0,0)$ is a point of order $2$ in $E(K)$ and
$P=(-b, ab) \in E(K)$.  If we put $g(x)=x^2+(a^2+2b)x+b^2$, then arguments of Example \ref{ex4}
show that $W$ is divisible by $2$ in $\mathcal{E}^{(4)}_{a,b}(K)$, and therefore $\mathcal{E}^{(4)}_{a,b}(K)$ contains a point of order $4$.
In fact, it contains exactly two points of order 4, namely, $(-b,ab)$ and $(-b,-ab)$.
\end{proof}

\begin{rem}
D. Kubert \cite{Kubert} described another family of elliptic curves
$$\mathcal{E}_{4,t}: y^2+xy-ty=x^3-tx^2$$
with point $Q=(0,0)$ of  order $4$.
The equation for $\mathcal{E}_{4,t}$ is equivalent to
$$ y^2+2y\frac{(x-t)}{2}+\frac{(x-t)^2}{4}=(x-t)x^2+\frac{(x-t)^2}{4}.$$
Now the change of variables
$$\tilde x =x-t, \ \tilde y=y+\frac{(x-t)}{2}$$
allows us to rewrite the equation as follows:
$$\tilde y^2=\tilde x \left((\tilde x+t)^2+\frac{\tilde x }{4}\right)=\tilde x\left(\tilde x^2+(2t+\frac{1}{4})\tilde x+t^2\right).$$
This implies that $\mathcal{E}_{4,t}$ is isomorphic to $\mathcal{E}_{a,b}$ with $b=t, a =1/2$.
\end{rem}

\begin{thm}\label{iso4}
Let $K$ be a field with $\fchar(K)\ne 2$, and let
%, $k_0$ a subfield of $K$ and
 $a,b$ be nonzero elements of $K$ such that $a^2+4b$ is not a square in $K$. (In particular, $a^2+4b\ne 0$.)
%that are algebraically independent over $k_0$ and $K=k_0(a,b)$.
Let $c,d$ be nonzero elements of $K$ such that $c^2+4d \ne 0$. Then the elliptic curves $\mathcal{E}_{a,b}$ and $\mathcal{E}_{c,d}$
are $K$-isomorphic if and only if there exists a nonzero $u \in K$ such that
$$c^2+2d=u^2(a^2+2b), \ u^4 d^2=b^2.$$
\end{thm}

\begin{proof}
Since  $a^2+4b$ is not a square in $K$, the quadratic polynomial $x^2+(a^2+2b)x+b^2$ has {\sl no} roots in $K$.
This implies that $\mathcal{E}_{a,b}$ has exactly one point of order $2$, namely, $W_{a,b}=(0,0)$.
Let $\phi:\mathcal{E}_{a,b}\cong \mathcal{E}_{c,d}$
be a $K$-isomorphism. Since $W_{c,d}=(0,0)$ is a point of order $2$ in $\mathcal{E}_{c,d}(K)$, it is the only point of order $2$ in
 $\mathcal{E}_{c,d}(K)$ and $\phi$ sends $W_{a,b}=(0,0)\in \mathcal{E}_{a,b}$ to $W_{c,d}=(0,0)\in \mathcal{E}_{c,d}$. Let
 $${\tilde y}^2={\tilde x }\left({\tilde x }^2+(c^2+2d){\tilde x }+d^2\right)$$
be the equation of  $\mathcal{E}_{c,d}$ in coordinates ${\tilde x},{\tilde y}$.
 According to \cite[Sect. 1, p. 301, formula (1.13)]{Tate}, there exist a nonzero $u\in K$ and $\alpha,\beta,\gamma \in K$ such that the following conditions hold:
$$
{\tilde x }\circ\phi=u^2 x+\alpha, \ {\tilde y }\circ\phi=u^3 y+\beta u^2 x+\gamma .$$
Since $\phi$ sends $(0,0) \in \mathcal{E}_{a,b}(K)$ to $(0,0) \in \mathcal{E}_{c,d}(K)$, we have $\alpha=0=\gamma$.
Since $\phi$ commutes with the multiplication by $-1$,  we conclude that $\beta=0$, i.e.,
$\phi$ is defined by the formulas
\begin{equation}
\label{phiTate}
{\tilde x }\circ\phi=u^2 x, \ {\tilde y }\circ\phi=u^3 y.
\end{equation}
Now \cite[Sect. 1, p. 301, formula (1.14)]{Tate} implies that
\begin{equation}
\label{abcdu}
u^2 (a^2+2b)=(c^2+2d), \ u^4 b^2=d^2.
\end{equation}
This proves our assertion in one direction. In order to prove it in the opposite direction, let us assume  that \eqref{abcdu} holds. Then the formulas \eqref{phiTate} define an isomorphism
$\phi: \mathcal{E}_{a,b}\cong \mathcal{E}_{c,d}$.
\end{proof}

\begin{rem}
Let us look more closely at   equalities \eqref{abcdu}. They may be rewritten as follows:
$$u^2 (a^2+2b)=(c^2+2d), \ d = \pm u^2 b.$$
It follows that
$u^2 (a^2+2b)=(c^2\pm 2 u^2 b)$, i.e., either
$$d=u^2 b, \ u^2 (a^2+2b)=(c^2+ 2u^2 b),$$
and so
$$u^2 a^2=c^2, \ c=\pm ua,$$ or
$$d=-u^2 b, \ u^2 (a^2+2b)=(c^2-2u^2 b),$$
and therefore
$$u^2 (a^2+4b)=c^2;$$
in the latter case $a^2+4b$ is a square in $K$, which is not the case. So, we have
$$d=u^2 b, \ dc=\pm ua,$$
 and therefore $u=\pm c/a$, which implies that
$$b=\frac{d a^2}{c^2}.$$
We have
$$(c^2+2d)=\left(\pm \frac{c}{a}\right)^2\left(a^2+2\cdot \frac{d a^2}{c^2}\right)u^2= u^2 (a^2+2b).$$
It follows that if $a$ is a fixed nonzero element of $K$ and $E$ is an elliptic curve over $K$ such that $E(K)$ contains  a point of order $4$ and exactly one point of order $2$, then there exists exactly one (nonzero) $b \in K$ such that $E$ is $K$-isomorphic to $\mathcal{E}_{a,b}$. In particular, if we take $a=1/2$, then we get that there is exactly one $t\in K$ such that  $E$ is $K$-isomorphic to $\mathcal{E}_{4,t}$.
\end{rem}
\section{Elliptic curves with points of order 8}
Let $E$ be an elliptic curve over a field $K$ with $\fchar(K)\ne 2$. Suppose that $E(K)$ contains exactly one point of order 2 and two points of order 4.
Then there exist {\sl nonzero} $a,b \in K$ such that $a^2+4b$  is {\sl not} a square in $K$ and $E$ is $K$-isomorphic to
$$\mathcal{E}^{(4)}_{a,b}: y^2=x\left(x^2+(a^2+2b)x+b^2\right).$$

\begin{sect}
\label{Q84}
Recall that $\mathcal{E}^{(4)}_{a,b}(K)$ contains exactly two points of order 4, namely, $Q_{+}=(-b,ab)$ and  $Q_{-}=(-b,-ab)=-Q_{+}$.
It follows  that  the existence of a point of order 8 in $\mathcal{E}^{(4)}_{a,b}(K)$  is equivalent to the divisibility by $2$ of  $Q_{+}$ in
$\mathcal{E}^{(4)}_{a,b}(K)$. By Theorem \ref{divideS} and  Remark \ref{nonvan},    $Q_{+}$ is divisible by 2 if and only if $(-b) -0$ is a square in $K$ and there exists $r \in K$
such that  $r^2=x_0 -\alpha$ and {\sl nonzero} $(2x_0+p)(x_0-\alpha)-2y_0 r$
is a  square in $K$, where
$$\alpha=0, p=a^2+2b; \ x_0=-b, y_0=ab.$$
This means that $r^2=-b \ne 0$ and
$$(-2b+a^2+2b)(-b)-2ab r=a^2 r^2+2 a r^2 r=r^2 (a^2+2 a r)=r^2\left((a+r)^2-r^2\right)$$
is a nonzero square in $K$.
It follows that  $Q_{+}$ is divisible by 2 in  $\mathcal{E}^{(4)}_{a,b}(K)$ if and only if there exist nonzero $r, T \in K$ such that
\begin{equation}
\label{rt}
b=-r^2, \ T^2=(a+r)^2-r^2.
\end{equation}

If this is the case,  then  Theorem \ref{divideS}
gives us two distinct halves of  $Q_{+}$ in $\mathcal{E}^{(4)}_{a,b}(K)$, namely,

\begin{equation}
\label{ab8}
\mathcal{Q}_{r,T}=(x_1,y_1):= \left(-b+rT-\frac{ab}{r}, -ab-\left(r+T\right)\left(rT -\frac{ab}{r}\right)\right)
\end{equation}
$$\mathcal{Q}_{r,-T}=(x_{-1},y_{-1}):=\left(-b-rT-\frac{ab}{r}, -ab-\left(r-T\right)\left(-rT  -\frac{ab}{r}\right)\right);$$
$$2 \mathcal{Q}_{r,T}=2\mathcal{Q}_{r,-T}=Q_{+}, \ \mathcal{Q}_{r,-T}= \mathcal{Q}_{r,T}+W_3.$$
This implies that both $\mathcal{Q}_{r,T}$ and $\mathcal{Q}_{r,-T}$ have order 8.
\end{sect}

\begin{prop}
\label{ord8}
$\mathcal{E}^{(4)}_{a,b}(K)$ contains a point of order $8$ if and only if there exist nonzero $r,t \in K$   such that
$t \ne 0,\pm 1$  and
$$b=-r^2,  \ a=\frac{2t^2r}{1-t^2}.$$
\end{prop}

\begin{proof}
Dividing the second equation in \eqref{rt} by $r^2$, we get
$$\left(\frac{T}{r}\right)^2=\left(\frac{a}{r}+1\right)^2-1.$$
Using the well-known rational parametrization
of the hyperbola, we obtain that this equality is equivalent to the existence of a  $t\in K, t\ne\pm1,$ such that
$$\frac{a}{r}+1=\frac{1+t^2}{1-t^2},\; \frac{T}{r}=\frac{2t}{1- t^2}.$$
This means that
\begin{equation}
\label{au8}
a=\frac{2t^2r}{1-t^2}, \;  T=\frac{2tr}{1-t^2}
\end{equation}
for  $t\in K$, $t\ne\pm1$. Since $T\ne 0$, we have $t\ne0$.
  We still need to find what does it mean in terms of $r,t$ that  $a^2+4b$ is {\sl not} a square in $K$. We have
$$a^2+4b=\left(\frac{2t^2r}{1-t^2}\right)^2+4 (-r^2)=4r^2\frac{2t^2-1}{(1-t^2)^2}.$$
This implies that $a^2+4b$ is not a square in $K$ if and only if $t \ne \pm 1$  and $2t^2-1$ is not a square.
\end{proof}

\begin{thm}
\label{order8}
Let $E$ be an elliptic curve over a field $K$ with $\fchar(K)\ne 2$. Suppose that $E(K)$ contains exactly one point of order $2$.
Then the  following conditions are equivalent.
\begin{itemize}
\item[(i)]
$E(K)$ contains a cyclic subgroup of order $8$.
\item[(ii)]
There exists  $t \in K$ such that
\begin{itemize}
\item[(1)]
$t  \ne0, \pm 1$,
\item[(2)]
 $2t^2-1$ is not a square in $K$,
 \item[(3)]
 $E$ is $K$-isomorphic to the elliptic curve
$$\mathcal{E}^{(8)}_{t}: y^2=x\left(x^2+2\frac{t^4+2t^2-1}{ (t^2-1)^2}x+1\right).$$
\end{itemize}
\end{itemize}
\end{thm}

\begin{proof}
Combining Theorem \ref{order44} with Proposition \ref{ord8}, we conclude that the condition (i) is equivalent to the existence
of nonzero $r,t \in K$, $t\ne\pm1$, such that

$$a=\frac{2t^2r}{1-t^2}, b=-r^2$$
enjoy the following properties.
\begin{itemize}
\item[(a)]
$a \ne 0$.
\item[(b)]
$a^2+4b$ is {\sl not} a square in $K$.
\item[(c)]
$E$ is isomorphic over $K$ to the elliptic curve
$$\mathcal{E}^{(4)}_{a,b}: y^2=x\left(x^2+(a^2+2b)x+b^2\right).$$
\end{itemize}
We notice  that $a=0$ if and only if $t=0$; moreover,
%$$a^2+4b=\left(\frac{r}{2t}\right)^2 ((t-1)^4-16 t^2),$$
$$b^2=(-r^2)^2=r^4,$$
$$a^2+2b=2r^2\frac{(t^4+2t^2-1)}{(t^2-1)^2}.$$
Consequently, the curve $E$ is $K$-isomorphic to the curve
$$\mathcal{E}^{(8)}_{r,t}: y^2=x\left(x^2+2r^2\frac{t^4+2t^2-1}{ (t^2-1)^2}x+r^4\right).$$
 Dividing both sides by $r^6$ and using the change of variables $\tilde x=x/r^2$, $\tilde y=y/r^3 $, we get
 the elliptic curve
 $$\mathcal{E}^{(8)}_{t}: \tilde y^2=\tilde x\left(\tilde x^2+2\frac{t^4+2t^2-1}{ (t^2-1)^2}\tilde x+1\right), $$
(where $t \ne 0, \pm 1$) that is $K$-isomorphic to $E$. It remains to notice that  $a^2+4b$ is {\sl not} a square in $K$ if and only if
 $2t^2-1$ is not a square.
\end{proof}

\begin{ex}
Let us describe explicitly points of order 8 in $\mathcal{E}^{(8)}_{t}(K)$. Notice that
$\mathcal{E}^{(8)}_{t}=\mathcal{E}^{(4)}_{a,b}$ with
$$a=\frac{2t^2}{1-t^2}, \ b=-1, r=1, T=\frac{2t}{1-t^2}.$$
In particular, $\mathcal{E}^{(8)}_{t}(K)$ contains a point
$Q_{+}=(-b,ab)=(1, -\frac{2t^2}{1-t^2})$ of order 4.
Clearly, the conditions \eqref{rt} hold. Then  \eqref{ab8} give us two distinct order 8 points in $\mathcal{E}^{(8)}_{t}(K)$, namely,

$$\mathcal{Q}_{r,T}= \left(1+\frac{2t}{1-t^2}+\frac{2t^2}{1-t^2}, \frac{2t^2}{1-t^2}-\left(1+\frac{2t}{1-t^2}\right)\left(\frac{2t}{1-t^2} +\frac{2t^2}{1-t^2}\right)\right)$$
$$=\left(\frac{t+1}{1-t},\frac{-2t}{(1-t)^2}\right),$$
$$\mathcal{Q}_{r,-T}=\left(1-\frac{2t}{1-t^2}+\frac{2t^2}{1-t^2}, \frac{2t^2}{1-t^2}-\left(1-\frac{2t}{1-t^2}\right)\left(-\frac{2t}{1-t^2} +\frac{2t^2}{1-t^2}\right)\right)$$
$$=\left(\frac{1-t}{1+t},\frac{2t}{(1+t)^2}\right);$$
$$2 \mathcal{Q}_{r,T}=2\mathcal{Q}_{r,-T}=Q_{+}, \ \mathcal{Q}_{r,-T}= \mathcal{Q}_{r,T}+W_3.$$
So, $\mathcal{E}^{(8)}_{t}(K)$ contains exactly four points of order 8, namely,
$$\left(\frac{1+t}{1-t},\frac{-2t}{(1-t)^2}\right), \ \left(\frac{1-t}{1+t},\frac{2t}{(1+t)^2}\right)$$
and their negatives
$$\left(\frac{1+t}{1-t},\frac{2t}{(1-t)^2}\right), \ \left(\frac{1-t}{1+t},\frac{-2t }{(1+t)^2}\right).$$
\end{ex}

\begin{rem}
D. Kubert \cite{Kubert} described another family of elliptic curves
\begin{equation}
\label{Kubert8}
\mathcal{E}_{8,d}:y^2+(1-c)xy-by=x^3-bx^2
\end{equation}
with point $Q=(0,0)$ of  order $8$,
where
$$c=\dfrac{(2d-1)(d-1)}{d},\quad b=(2d-1)(d-1).$$
Completing the square on the left-hand side of the  equation, we get
$$\left(y+\dfrac{(1-c)x-b}{2}\right)^2=x^3-bx^2+\left(\dfrac{(1-c)x-b}{2}\right)^2.$$
The curve given by the latter equation is isomorphic to
$$y^2=x^3+\left(\left(\dfrac{1-c}{2}\right)^2-b\right)x^2-\dfrac{(1-c)b}2x+\dfrac{b^2}4.$$
Since
$$\dfrac{1-c}{2}=\dfrac{-2d^2+4d-1}{2d},$$
we get
$$\begin{aligned}y^2=x^3+\left(\dfrac{-4d^4-4d^3+16d^2-8d+1}{4d^2}\right)x^2\\-
\dfrac{(-2d^2+4d-1)(2d-1)(d-1)}{2d}x+\dfrac{(2d-1)^2(d-1)^2}4.\end{aligned}$$
Multiplying both sides of the above equation by $d^6$ and replacing $yd^3$ by $y$ and $xd^2$ by $x$,
\noindent we get the equation with polynomial coefficients
$$\begin{aligned}y^2=x^3+\left(\dfrac{-4d^4-4d^3+16d^2-8d+1}{4}\right)x^2\\-
\dfrac{(-2d^2+4d-1)(2d-1)(d-1)d^3}{2}x+\dfrac{d^6(2d-1)^2(d-1)^2}4.\end{aligned}$$

The polynomial on the right can be factored as follows:
$$(x-(d^4-d^3))\left(x^2+\dfrac{1-8d(d-1)^2}{4}x-\dfrac{d^3(d-1)(2d-1)^2)}4\right).$$
Now using the change of variables $ x-(d^4-d^3)\to x$, we obtain that
 $$y^2=x\left(x^2+\frac{8d^4-16d^3+16d^2-8d+1}{4}x+d^4(d-1)^4\right).$$
 Dividing both sides by $d^6(d-1)^6$ and using the change of variables  $\tilde x=x/d^2(d-1)^2$, \break $\tilde y=y/d^3(d-1)^3$, we get
 the elliptic curve
 $$\tilde y^2=\tilde x\left(\tilde x^2+\frac{8d^4-16d^3+16d^2-8d+1}{4d^2(d-1)^2}\tilde x+1   \right) .$$
%that is $K$-isomorphic to.
It is easy to verify that
 $$\frac{8d^4-16d^3+16d^2-8d+1}{4d^2(d-1)^2}=2\frac{t^4+2t^2-1}{ (t^2-1)^2},$$
where $t=2d-1$. This implies that $\mathcal{E}_{8,d}$ is isomorphic over $K$ to $\mathcal{E}^{(8)}_{t}$
with $t=2d-1$.
 \end{rem}

\begin{thm}
Let $K$ be a field with $\fchar(K)\ne 2$. Let $s, t\in K$ be  nonzero elements of $K$ such that $s, t\ne   \pm 1$ and $2t^2-1, 2s^2-1$ are not   squares in $K$. If $K$ does not contain a primitive $4$th root of $1$, then the elliptic curves $\mathcal{E}^{(8)}_{s}$ and $\mathcal{E}^{(8)}_{t}$
are $K$-isomorphic if and only if $s=\pm t$ or $s^2+t^2=2s^2t^2$. If $K$ contains a primitive $4$th root of $1$, then
$\mathcal{E}^{(8)}_{s}$ and $\mathcal{E}^{(8)}_{t}$
are $K$-isomorphic if and only if $s=\pm t$, $s^2+t^2=2s^2t^2$, or $s^4t^4+2s^2+2t^2=4s^2t^2+1$.
\begin{proof}
It follows from Theorem \ref{iso4}
that $\mathcal{E}^{(8)}_{s}$ and $\mathcal{E}^{(8)}_{t}$
are $K$-isomorphic if and only if there exists   $u\in K$ such that $u^4=1$ and
$$\frac{s^4+2s^2-1}{ (s^2-1)^2}=u^2\frac{t^4+2t^2-1}{ (t^2-1)^2}.$$
If $K$ does not contain a primitive 4th root of 1, then the latter equality is possible only if
$$\frac{s^4+2s^2-1}{ (s^2-1)^2}=\frac{t^4+2t^2-1}{ (t^2-1)^2},$$
which takes place if and only if one of the equalities  $s=\pm t$ or $s^2+t^2=2s^2t^2$ is valid.
If $K$ contains a primitive 4th root of 1, then we have
$$\frac{s^4+2s^2-1}{ (s^2-1)^2}=\pm\frac{t^4+2t^2-1}{ (t^2-1)^2},$$
which holds if and only if $s=\pm t$, $s^2+t^2=2s^2t^2$, or $s^4t^4+2s^2+2t^2=4s^2t^2+1$.

\end{proof}
\end{thm}

\section{Elliptic curves with point of order 6}
Let $K$ be a field with $\fchar{K} \ne 0$. Let $E$ be an elliptic curve over $K$ defined by the equation
$$y^2=f(x),$$
where $f(x) \in K[x]$ is a monic cubic polynomial without repeated roots.

\begin{ex}
\label{exT6}
Let $t\in K \setminus{0,-4, 1/2}$. Let us consider the elliptic curve
$$\mathcal{E}^{(6)}_{t}: y^2=(x^2+(t^2+2t)x+t^2)(x+1)$$
over $K$. (We assume that $t\not\in\{0,-4, 1/2\}$  to  exclude the case when the cubic polynomial has a repeated root.) The group
$E(K)$ contains the point $W_3=(-1,0)$ of order 2 and the point $P=(0,t)$. Let us put
$$y_0=t, r=1, T=t.$$
Then the curve $\mathcal{E}^{(6)}_{t}$ coincides with the elliptic curve \eqref{ellTr} from Theorem \ref{divRed2}.
According to this theorem,
$$\mathcal{Q}_{r,T}=\left(r T-\frac{y_0}{r}, -y_0-(r+T)\left(r T-\frac{y_0}{r}\right)\right)$$
  and
 $$\mathcal{Q}_{r,-T}=\left(-r T-\frac{y_0}{r}, -y_0-(r-T)\left(-r T-\frac{y_0}{r}\right)\right)$$
are distinct points of  $E(K)$ that satisfy
$$2 \mathcal{Q}_{r,T}=2 \mathcal{Q}_{r,-T}=P,   \  \mathcal{Q}_{r,-T}=\mathcal{Q}_{r,T}+W_3 .$$
In our case
$$r T=t,  \frac{y_0}{r}=t, r T-\frac{y_0}{r}=0,  -r T-\frac{y_0}{r}=-2t.$$
This implies that
$$\begin{aligned}\mathcal{Q}_{r,T}=(0,-y_0)=(0,-t)=-P, \  \mathcal{Q}_{r,T}+W_3=\mathcal{Q}_{r,-T}\\=(-2t,-t-(1-t)(-2t))=(-2t, t-2t^2).\end{aligned}$$
Since
$$\mathcal{Q}_{r,T}=-P, \ 2 \mathcal{Q}_{r,T}=-P,$$
$P$ and $\mathcal{Q}_{r,T}$ have order 3 while $\mathcal{Q}_{r,-T}=\mathcal{Q}_{r,T}+W_3$ has order 6.
So, the point $(0,t)\in \mathcal{E}^{(6)}_{t}(K)$ has order 3 and the point $(-2t, t-2t^2)\in \mathcal{E}^{(6)}_{t}(K)$ has order 6.

Notice that $\mathcal{E}^{(6)}_{t}(K)$  contains exactly one point of order 2 if and only if the discriminant  $(t^2+2t)^2-4t^2=t^2(t^2+4t)$  of $x^2+(t^2+2t)x+t^2$ is {\sl not} a square, i.e., if and only if $t^2+4t$ is {\sl not} a square.

\end{ex}

\begin{thm}
\label{family6}
The following conditions on $E$ are equivalent.

\begin{itemize}
\item[(i)]
$E(K)$ contains a point of order $6$.
\item[(ii)]
There exists  $t \in K\setminus \{0,-4,1/2\}$ such that  $E$ is isomorphic over $K$ to the elliptic curve
$$\mathcal{E}^{(6)}_{t}: y^2=(x+1)(x^2+(t^2+2t)x+t^2).$$
\end{itemize}
\end{thm}

\begin{proof}
Suppose  $E(K)$ contains a point of order 6. This means that $E(K)$ contains a point of order $2$ and a point of order $3$.
The existence of a point of order 2 in $E(K)$ means that $f(x)$ has a root in $K$ say, $\alpha$ and one may represent $f(x)$ as a product
$$f(x)=g(x)(x-\alpha) \in K[x],$$
where
$$g(x)=x^2+px+q \in K[x]$$
is a monic quadratic polynomial without repeated roots such that our $\alpha\in K$ is {\sl not} a root of $g(x)$.
Then $W_3=(\alpha,0)\in E(K)$ is a point of order 2.

Let $P=(x_0,y_0)\in E(K)$ be a point of order 3.  Since $3 \ne 2$, we have  $y_0 \ne 0$. Using the change of variables $x \to x-x_0$ (and replacing $\alpha$ by $\alpha-x_0$ and $g(x)$ by $g(x-x_0)$), we may and will assume that $x_0=0$,
i.e., $P=(0,y_0)$. Since $P$ has order 3, it lies in $2 E(K)$, because $2 (-P)=P$. Let us apply the last assertion of Theorem \ref{divRed2} to $P=(0,y_0)$ and $Q=-P=(0,-y_0)$. We obtain
that there exist $r\in K$ and a nonzero $T \in K$ such that
$$r^2=-\alpha, \ g(x)=x^2+px+q=x^2+\left(T^2+2 \frac{y_0}{r}\right)x+\left(\frac{y_0}{r}\right)^2,$$
$$(0,-y_0)=Q=\mathcal{Q}_{r,T}=\left(r T-\frac{y_0}{r}, -y_0-(r+T)\left(r T-\frac{y_0}{r}\right)\right).$$
Looking at the $x$-coordinates, we see that
$$r T-\frac{y_0}{r}=0, \ rT=\frac{y_0}{r}, \ y_0=r^2 T,$$
and therefore
$$g(x)=x^2+(T^2+2rT)x+(rT)^2, \ f(x)=(x^2+(T^2+2rT)x+(rT)^2)(x+r^2)$$
and the equation for $E$ is
$$E: y^2=f(x)=(x^2+(T^2+2rT)x+(rT)^2)(x+r^2).$$
Dividing both sides of this equation  by $r^6$, and making a change of variables $\tilde x=x/r^2, \tilde y=y/r^3$, we obtain that $E$
is isomorphic to the elliptic curve
$$\mathcal{E}^{(6)}_{t}:\tilde y^2=(\tilde x+1)(\tilde x^2+(t^2+2t)\tilde x+t^2)$$
with $t=T/r \ne 0$. Since the  polynomial $(\tilde x+1)(\tilde x^2+(t^2+2t)\tilde x+t^2)$ has no multiple roots, we conclude that
 $t\not\in\{-4,0,1/2\}$.

The converse assertion follows from Example \ref{exT6}.
\end{proof}

\begin{rem}
In Theorem \ref{family6} we do {\sl not} assume that $\fchar(K)\ne 3$!
\end{rem}

\begin{ex}
Let $K=\mathbb{F}_3$ be the 3-element field. Then there is exactly one element $t$ in $\mathbb{F}_3\setminus \{0,-4\}$,
namely $t=1$. It follows from Theorem \ref{family6} that $E$ is an elliptic curve over $K$ such that $E(\mathbb{F}_3)$ contains a point of order 6 if and only if
 $E$ is $\mathbb{F}_3$-isomorphic to
 $$\mathcal{E}^{(6)}_{1}: y^2=(x+1)(x^2+1).$$
 Let us consider the curve $\mathcal{E}^{(6)}_{1}$. It contains an $\mathbb{F}_3$-point of order 6, namely,
 $Q=(-2t,  t-2t^2)=(-2,-1)$.
  Hence the whole group $\mathcal{E}^{(6)}_{1}(\mathbb{F}_3)$ has order  divisible by 6.
 On the other hand, by Hasse's bound, the order of $\mathcal{E}^{(6)}_{1}(\mathbb{F}_3)$ does not exceed $3+2 \sqrt{3}+1<6 \cdot 2$.
 This implies that $\mathcal{E}^{(6)}_{1}(\mathbb{F}_3)$ has order  6 and therefore coincides with the cyclic group of order 6 generated by $Q$.
\end{ex}

\begin{rem}
D. Kubert described another family of elliptic curves
$$\mathcal{E}_{6,t}: y^2+(1-c)xy-(c+c^2)y=x^3-(c+c^2)x^2$$
with point $Q=(0,0)$ of  order $6$.
The equation for $\mathcal{E}_{6,c}$ is equivalent to
$$ \begin{aligned} y^2+2y\frac{(1-c)x-(c+c^2)}{2}+\frac{((1-c)x-(c+c^2))^2}{4}\\
=x^3-(c+c^2)x^2+\frac{((1-c)x-(c+c^2))^2}{4}.\end{aligned}$$
The left-hand side is equal to $(y+ ({(1-c)x-(c+c^2)})/{2})^2$ while the right-hand side splits into the product
$$(x-c)\left(x^2+\frac{-3c^2-2c+1}{4}x-\frac{c^3+2c^2+c}{4}\right).$$
%$$((x-c-1)+1) \left((x-c-1)^2+\left(2(c+1)+\frac{-3c^2-2c+1}{4}\right)(x-c-1)+\left((c+1)^2+\frac{(-3c^2-2c+1)(c+1}{4}-\frac{c^3+2c^2+c}{4}\right)\right)$$
After the change of variables $\tilde x =x-(c^2+c)$,\ $\tilde y=y+({(1-c)x-(c+c^2)})/{2}$, we get the equation
$$\tilde y ^2=(\tilde x +c^2)\left(\tilde x ^2+\frac{5c^2+6c+1}{4}\tilde x +\frac{c^2(c+1)^2}{4}\right).$$
Dividing both sides by $c^6$,  we get the equation
$$\left(\frac{\tilde y}{c^3}\right)^2=\left(\frac{\tilde x}{c^2}+1\right)\left(\left(\frac{\tilde x}{c^2}\right)^2+\frac{5c^2+6c+1}{4c^2}\left(\frac{\tilde x}{c^2}\right)+\frac{(c+1)^2}{4c^2}\right).$$
Now the change of variables
 $\bar x=\tilde x /c^2, \ \bar y=\tilde y/c^3$ gives us the equation
$$\bar y^2=(\bar x+1)\left(\bar x^2+\frac{5c^2+6c+1}{4c^2}\bar x+\left(\frac{c+1}{2 c}\right)^2\right),$$
which is nothing else but the equation of $\mathcal{E}^{(6)}_{t}$ with $t=(c+1)/2c$.
\end{rem}

\section{Elliptic curves with point of order 12}

Let $K$ be a field with $\fchar(K)\ne 2$, and let
$t\in K\setminus \{0,-4,1/2\}$ be such that $t^2+4t$ is not a square in $K$. Let us consider the elliptic curve
$$E:=\mathcal{E}_t^{(6)}: y^2=g(x) (x+1),$$
where
$$g(x)=x^2+(t^2+2t)x+t^2$$
is a quadratic irreducible polynomial over $K$. Then $W=W_3=(-1,0)$ is the only point of order 2 in $E(K)$.  We know that $E(K)$ contains a point of order $3$. Hence $E(K)$ contains a point of order 12 if and only if $W$ is divisible by 2 in $E(K)$. This is equivalent to the condition that
$-1-\mathbf{x}$ is a square in the quadratic extension
$$K_g=K[x]/g(x)K[x]$$
of $K$, i.e., there exist $u,v \in K$ such that
$$-1-\mathbf{x}=(u \mathbf{x}+v)^2 \in K_g.$$
Clearly,  such $u \ne 0$.
(Here $\mathbf{x}$ is the image of $x$ in $K_g$.) In other words, $W$ is divisible by 2 in $E(K)$ if and only if there exist  $u,v \in K$ such that $(ux+v)^2 - (-1-x)$ is divisible by $g(x)$. The latter condition means that there exist  $u,v \in K$ such that
$$u^2 x^2+(2uv+1)x +(v^2+1)=(ux+v)^2+x+1$$
coincides with
$$u^2 g(x)=u^2 x^2+u^2 (t^2+2t)x+u^2 t^2,$$
i.e.,
$$2uv+1=u^2 (t^2+2t), \  v^2+1=u^2 t^2.$$
Subtracting one equation from the other, we obtain that
$$v^2-2uv=-2t u^2.$$
Dividing both sides by $u^2$ and putting $\lambda=v/u$, we get
$$\lambda^2-2\lambda=-2t,$$
which gives us
$$t=\frac{2\lambda-\lambda^2}{2}, \ v=\lambda u.$$
This implies that $\lambda \ne 0$, since $t \ne 0$ and therefore $v \ne 0$. We have
$$(\lambda u)^2+1=u^2 \left(\frac{2\lambda-\lambda^2}{2}\right)^2,$$
which means that
$$1=u^2\left( \left(\frac{2\lambda-\lambda^2}{2}\right)^2-\lambda^2\right)=
u^2\lambda^2 \left( \left(\frac{2-\lambda}{2}\right)^2-1\right).$$
Consequently,
$$1=v^2 \left (\left(1-\frac{\lambda}{2}\right)^2-1\right).$$
Putting $s=1-\lambda/2$ and  $\mu=1/v$, we get the equation for a hyperbola in $(s,\mu)$-coordinates.
$$\mu^2=s^2-1.$$
Using the standard parametrization
$$s=\frac{1+T^2}{1-T^2},\quad  \mu=\frac{2T}{1-T^2} $$
of this hyperbola,
we obtain
$$\begin{aligned}\lambda=\frac{4T^2}{T^2-1}, \quad t=\frac{2\lambda-\lambda^2}{2}=-\frac{4T^2(T^2+1)}{(T^2-1)^2}, \\
 t^2+2t=\frac{8T^2(T^2+1)(T^4+4T^2-1)}{(T^2-1)^4},\end{aligned}$$
and the equation of $E$ takes the form
$$y^2=(x+1)\left(x^2+\frac{8T^2(T^2+1)(T^4+4T^2-1)}{(T^2-1)^4}x+\frac{16T^4(T^2+ 1)^2}{(T^2-1)^4}\right).$$
The condition $t\in K\setminus \{0,-4,1/2\}$ is equivalent to $T\not\in\{0,\pm1,\pm\sqrt{-1}\}$, $3T^2-1\neq0$, and
$3T^2+1\neq0$. Since
$$t^2+4t=\frac{16T^2(T^2+1)(3T^2-1)}{(T^2-1)^4},$$
$t^2+4t$ is not a square in $K$ if and only if $(T^2+1)(3T^2-1)$ is not a square in $K$.

We have proved the following statement.
\begin{thm}
\label{family12}
The following conditions on $E$ are equivalent.

\begin{itemize}
\item[(i)]
$E(K)$ contains exactly one point of order $2$ and a point of order $12$.
\item[(ii)]
There exists  $T \in K\setminus\{0,\pm1,\pm\sqrt{-1}\}$ such that  $3T^2+1\neq0$, $(T^2+1)(3T^2-1)$ is not a square, and $E$ is isomorphic over $K$ to the elliptic curve
$$\mathcal{E}^{(12)}_{T}: y^2=(x+1)\left(x^2+\frac{8T^2(T^2+1)(T^4+4T^2-1)}{(T^2-1)^4}x+\frac{16T^4(T^2+ 1)^2}{(T^2-1)^4}\right).$$
\end{itemize}
\end{thm}
\begin{rem}
$\mathcal{E}^{(12)}_{T}(K)$ contains points $(0, \pm4T^2(T^2+1)/(T^2-1)^2)$ of order 3
and exactly two points of order 4.
Applying the explicit formulas from Theorem \ref{divExp} to $(x_0,y_0)=(-1,0)$ and taking into account that $r=x_0-\alpha=0$, $\rr =u\mathbf x+v$,
$$\begin{aligned}\Norm (\rr) &=u^2\Norm (\mathbf x)+uv\Tr (\mathbf x)+v^2=
t^2 u^2-(t^2+2t)uv+v^2,\\
\Tr(\rr)&=u\Tr(\mathbf x)+2v=-(t^2+2t)u+2v,
\end{aligned}
$$ we get
the following two points of order $4$ on $\mathcal{E}^{(12)}_{T}$:
$$
(t^2 u^2-(t^2+2t)uv+v^2, (t^2 u^2-(t^2+2t)uv+v^2)(-(t^2+2t)u+2v)),
$$
$$
(t^2 u^2-(t^2+2t)uv+v^2, -(t^2 u^2-(t^2+2t)uv+v^2)(-(t^2+2t)u+2v)),
$$
where
$$\begin{aligned}t=-\frac{4T^2(T^2+1)}{(T^2-1)^2}, \  t^2+2t=\frac{8T^2(T^2+1)(T^4+4T^2-1)}{(T^2-1)^4}, \ \\v=\frac {1}{\mu}= \frac{1-T^2}{2T}, \ u=\frac {v}{\lambda}=-\frac{(T^2-1)^2}{8T^3}.\end{aligned}$$ Substituting these expressions for $t$, $u$, and $v$ in the above formulas, we obtain the order $4$ points
$$\left(-\frac{3T^2+1}{T^2-1}, -\frac{8 T^3(3T^2+1)}{(T^2-1)^3}\right), \ \left(-\frac{3T^2+1}{T^2-1}, \frac{8 T^3(3T^2+1)}{(T^2-1)^3}\right).$$

\end{rem}
\begin{rem}
In Theorem \ref{family12} we do {\sl not} assume that $\fchar(K)\ne 3$!
\end{rem}
\section{Elliptic curves with rational points of order 10}

Let $E$ be an elliptic curve over a field $K$ with $\fchar(K)\ne 2$. Suppose that $E(K)$ contains exactly one point of order 2 and a point $P$ of order $5$. We may assume that the first coordinate of $P$ is 0.
Since $P$ is divisible by $2$, Theorem \ref{divRed2} and Remark \ref{req1} tell  us that $E$  is $K$-isomorphic to the elliptic curve
\begin{equation}
\label{ellt10}
{E}_{t,{y}_0}:{y}^2=({x}+1)\left({x}^2+\left(t^2+2{y}_0\right){x}+{y}_0^2\right),
\end{equation}
where $t^2+4y_0\neq0$ and $1-t^2-2y_0+y_0^2\neq0$.
We may also assume
 that $P=(0,y_0)$.
Clearly $P$ has order $5$ if and only if there exists a $Q\in E(K)$ such that $2P=-Q$ and $2Q=P$.
Using equation \eqref{doubleX} with $p=t^2+2{y}_0$ and $q={y}_0^2$ and the equality $x(Q)=t-y_0$ obtained in
 Remark \ref{req1}, we can write the equalities $2P=-Q$ and $2Q=P$ in the following equivalent form:
 $$
 \frac{(t^2+2y_0-y_0^2)^2}{4y_0^2}-1=t-y_0.$$
 Multiplying both sides by $y_0^2$ and removing parentheses, we get
 $$t^4+y_0^4+4t^2y_0-4ty_0^2-2t^2y_0^2=0,$$
 $$(t^2-y_0^2)^2+4ty_0(t-y_0)=0.$$
 If $t=y_0$, then by Remark \ref{req1}, we obtain $Q_{1,t}=(0,-y_0)$,  $P=-Q_{1,t}$,  and so $3P=0$, which is impossible. Dividing both sides of the above equation by $t-y_0$, we get
 $$(t-y_0)(t+y_0)^2+4ty_0=0.$$ To obtain a rational parametrization of this equation, let us put $y_0=ut$ with $u\in K\setminus\{0,1\}$. Then
 $$t^3(1-u)(1+u)^2+4t^2u=0,$$
 hence
 $$t=\frac{4u}{(u-1)(u+1)^2}, \ y_0=\frac{4u^2}{(u-1)(u+1)^2}.$$
 Since
 $$
 t^2+2y_0=\frac{16u^2}{(u-1)^2(u+1)^4}+\dfrac{8u^2}{(u-1)(u+1)^2}=\dfrac{8u^2(u^3+u^2-u+1)}
 {(u-1)^2(u+1)^4},
 $$
  equation \eqref{ellt10} transforms to
 $$y^2=(x+1)\left(x^2+\dfrac{8u^2(u^3+u^2-u+1)}
 {(u-1)^2(u+1)^4}x+\frac{16u^4}{(u-1)^2(u+1)^4}\right).$$
 Since
 $$t^2+4y_0=\frac{16u^3(u^2+u-1)}{(u-1)^2(u+1)^4}\;\text{and}\;   1-t^2-2y_0+y_0^2=\frac{(u^2-4u-1)(u-1)
}{(u+1)^3},$$
 the conditions $t^2+4y_0\neq0$ and $1-t^2-2y_0+y_0^2\neq0$ are equivalent to
 $u\neq0, u\neq\pm1$, $u^2+u\neq1$, and  $u^2-4u-1 \neq0$. Moreover, $t^2+4y_0$ is not a square if and only if
 $u(u^2+u-1)$ is not a square.

We have proved the following statement.

\begin{thm}
\label{family10}
The following conditions on $E$ are equivalent.

\begin{itemize}
\item[(i)]
$E(K)$ contains exactly one point of order $2$ and a point of order $10$.
\item[(ii)]
There exists  $u \in K\setminus\{0,\pm1, (-1\pm\sqrt 5)/2, 2\pm\sqrt5\}$ such that  $u(u^2+u-1)$ is not a square and $E$ is isomorphic over $K$ to the elliptic curve
$$\mathcal{E}^{(10)}_{u}: y^2=(x+1)\left(x^2+\dfrac{8u^2(u^3+u^2-u+1)}
 {(u-1)^2(u+1)^4}x+\frac{16u^4}{(u-1)^2(u+1)^4}\right).$$
\end{itemize}
\end{thm}
\begin{rem}
$\mathcal{E}^{(10)}_{u}(K)$  contains   a point $\left(0,{4u^2}/{(u-1)(u+1)^2}\right)$ of order $5$
and exactly one point $(-1,0)$ of order $2$.
\end{rem}
\begin{rem}In Theorem \ref{family10} we do {\sl not} assume that $\fchar(K)\ne 5$!
\end{rem}
\section{Elliptic curves in characteristic 2}
In this section we assume that $K$ is a field of characteristic 2 and $\bar{K}$ is its algebraic closure. It is known  \cite[Appendix A]{Silverman} that $E$ is {\sl ordinary} (i.e., $E(\bar{K})$ contains a point of order 2) if and only if $j(E)\ne 0$.
Let $E$ be an  elliptic curve over $K$ defined by the equation
$$y^2+xy=x^3+a_2 x^2 +a_6,$$
where
$$a_2, a_6 \in K; \ a_6 \ne 0, \  j(E)=\frac{1}{a_6}\ne 0$$
 As above, $E$ has the only one infinite point $\infty=(0:1:0)$, which is taken as the zero of the group law on $E$.
It is known \cite[Appendix A]{Silverman}  that  $E$ is  ordinary.
In addition, every ordinary elliptic curve $\tilde{E}$ over $K$ is isomorphic to $E$ for suitable
$a_2\in K$ and $a_6=1/j(\tilde{E})$
\cite[Appendix A]{Silverman}, \cite[Sect. 2.8]{Wash}.
If $P=(x_0,y_0) \in E(\bar{K})$, then
\begin{equation}
\label{involution}
-P=(x_0,y_0+x_0) \in E(\bar{K}).
\end{equation}
It follows that
$$W_3=(0,\sqrt{a_6})=\left(0, \frac{1}{\sqrt{j(E)}}\right)\in E(\bar{K})$$
is the only point of order 2 in $E(\bar{K})$.

This implies the following assertion (that may be also extracted from \cite{Kramer}).

\begin{prop}
An    elliptic curve  $E$ over a  field $K$ of characteristic $2$
has a $K$-rational point of order $2$ if and only if $j(E)$ is a nonzero square in $K$.
\end{prop}

Our first goal is to find explicitly both halves of $P$. So, let $Q=(x_1,y_1) \in  E(\bar{K})$ with $2 Q=P$.
Clearly, $Q \ne W$, i.e.,
$x_1 \ne 0$,
and therefore
 the tangent line $\mathcal L$ to $E$ at $Q$ is {\sl not} vertical, i.e, may be written in the form
$$y=l x+m; \ l,m \in \bar{K}.$$
(If $Q \in E(K)$, then both $l$ and $m$ lie in $K$.) Since $2Q=P$, the line $\mathcal L$ contains $-P$. Restricting (as usual) the equation of $E$ to $\mathcal L$,
we get the equation
$$g(x)=x^3+a_2 x +a_6+(l x+m)^2+x(lx+m)=0.$$
We know that $x=x_1$ is a multiple root of the monic cubic polynomial $g(x)$ and $x=x_0$ is a root of $g(x)$. This implies that
$$g(x)=(x+x_1)^2(x+x_0).$$
It follows that
$$x^3+(a_2+l^2+l)x^2+ m x+(a_6+m^2)=x^3+x_0 x^2+x_1^2 x+x_1^2 x_0,$$
i.e,
$$(a_2+l^2+l)=x_0, \ m=x_1^2, \ (a_6+m^2)=x_1^2 x_0.$$
This implies that
$$l^2+l=x_0+a_2.$$
Since $-P=(x_0,y_0+x_0)$ lies on $L: y=l x+m$, we obtain
\begin{equation}
\label{DIV2}
m=(y_0+x_0)-l x_0=y_0+(l+1)x_0 ,   \ x_1=\sqrt{m}, \ y_1=l x_1+m=l \sqrt{m}+m.
\end{equation}
If we replace $l$ by $l+1$, then we should replace $m$ by $m+x_0$ and  $x_1$ by $x_1+\sqrt{x_0}$.
We also obtain the following formulas:
\begin{equation}
\label{x1x0}
\sqrt{a_6}+m=x_1 \sqrt{x_0}, \  \sqrt{a_6}+y_0+(l+1)x_0=x_1 \sqrt{x_0}.
\end{equation}

\begin{ex}
\label{order42}
Suppose that $x_0=0$, i.e., $P$ is a point of order 2. Then
$$m=\sqrt{a_6}, \ x_1=\sqrt[4]{a_6}, \ y_1=l \sqrt[4]{a_6}+\sqrt{a_6}.$$
If, in addition, $a_2=0$, then $l=0$ or $1$, and we get two halves
$Q_1=(\sqrt[4]{a_6},\sqrt{a_6})$ and $Q_1=(\sqrt[4]{a_6},\sqrt[4]{a_6}+\sqrt{a_6})$
of $P$. This implies that $Q_1$ and $Q_2$ are (the only) points of order 4.
\end{ex}

The following assertion may be extracted from \cite[Prop. 1.1]{Kramer}.
\begin{thm}
\label{division22}
Suppose that $E(K)$ contains a point of order $2$, i.e., there exists $\beta \in K$ such that $a_6=\beta^2$.
Then a point $P=(x_0,y_0) \in E(K)$ is divisible by $2$ in $E(K)$ if and only if the following conditions hold.
\begin{itemize}
\item[(i)]
$x_0$ is a square in $K$, i.e., there exists $r \in K$ such that $r^2=x_0$.
\item[(ii)]
There exists $l \in K$ such that
$$l^2+l=x_0+a_2.$$
\item[(iii)]
If $x_0=0$, i.e., $P$ is a point of order $2$, then  $a_6$ is a fourth power in $K$.
\end{itemize}
\end{thm}

\begin{proof}
Let $P=(x_0,y_0)$ is a $K$-point on $E$.

Assume that $P\in 2 E(K)$, i.e., there exists  a point $Q=(x_1,y_1)$ on $E$ such that $2Q=P$ and $x_1,y_1 \in K$.
Then in the notation above $l,m \in K$, because the tangent line to $E$ at $K$-point $Q$ is defined over $K$. It
follows that (ii) holds.
If $x_0=0$, then it follows from Example \ref{order42} that conditions (i)-(iii) hold. So, we may assume that $x_0 \ne 0$.
It follows from \eqref{x1x0} that
$\beta+m=x_1 \sqrt{x_0}$. This implies that
$$\sqrt{x_0}=\frac{\beta+m}{x_1}$$
lies in $K$, i.e., (i) holds.

Now assume that conditions (i)-(iii) hold and  $Q=(x_1,y_1) \in E(\bar{K})$ satisfies $2Q=P$. If $x_0=0$, then
the explicit formulas of Example \ref{order42} tell us that $Q \in E(K)$. Suppose that $x_0 \ne 0$. Let $\mathcal L: y=lx+m$
be the equation of the tangent equation to $E$ at $Q$. The condition (ii) implies that $l \in K$. Since $\mathcal L$ contains
the $K$-point $-P$, $m$ also lies in $K$. Now  equation \eqref{x1x0} tells us that
$\beta+m =x_1 r$ with nonzero $r=\sqrt{x_0}\in K$, and therefore
$$x_1=\frac{\beta+m}{r}$$
also lies in $K$. This implies that $y_1=l x_1+m$ also lies in $K$.
\end{proof}

\begin{cor}
\label{gamma4}
An  elliptic curve  over a field $K$ of characteristic $2$ has a $K$-rational point of order $4$ if and only if there exists a nonzero $\gamma\in K$
such that $E$ is
 $K$-isomorphic to the elliptic curve
$$\mathbf{E}_{4,\gamma}: y^2+xy=x^3+\gamma^4.$$
\end{cor}

\begin{proof}
The result follows almost immediately from Theorem \ref{division22} combined with Example \ref{order42}.
Indeed, if an elliptic curve $E$ over $K$ has a $K$-point of order $4$, then it has a $K$-point of order $2$
and therefore is ordinary, i.e., is $K$-isomorphic to
$$y^2+xy=x^3+a_2 x^2 +a_6,$$
where
$$a_2, a_6 \in K; \ a_6 \ne 0, \  j(E)=\frac{1}{a_6}\ne 0$$
and $a_6$ is a square in $K$. In addition, $P=(0,\sqrt{a_6})\in E(K)$ is a point of order $2$ that
actually lies in $2E(K)$. It follows from Theorem \ref{division22} that there exist $\beta,l\in K$ such that
$$\beta^4=a_6, \ l^2+l=a_2.$$
One should only notice that the change of variables
$x \mapsto x, y \mapsto y+lx$ establishes a $K$-isomorphism between elliptic curves
$y^2+xy=x^3+a_2 x^2a_6$ and
$$y^2+xy=x^3+(a_2+l^2+l)x^2+a_6=x^3+\beta^4.$$
\end{proof}

\begin{cor}
\label{jE4}
Let $E$ be an elliptic curve over $K$ such that $E(K)$ contains a point of order $4$. Then $j(E)$ is a nonzero fourth power in $K$.

Conversely, if $c$ is a nonzero fourth power in $F$, then there exists an elliptic curve $E$ over $K$ such that $j(E)=c$ and $E(K)$ contains a point of order $4$.
Such an $E$ is unique up to $K$-isomorphism.
\end{cor}

\begin{proof}
Suppose $E(K)$ contains a point of order 4. By Corollary \ref{gamma4}, there exists a nonzero $\gamma \in K$ such that $E$ is $K$-isomorphic to
$$\mathbf{E}_{4,\gamma}: y^2+xy=x^3+\gamma^4.$$
Clearly,
$$j(E)=j(\mathbf{E}_{4,\gamma})=\frac{1}{\gamma^4}=\left(\frac{1}{\gamma}\right)^4.$$
This proves the first assertion.

Conversely, if $\delta\in K$ satisfies $\delta^4=c$, then  we put $\gamma=1/\delta$ and consider the elliptic curve $\mathbf{E}_{4,\gamma}$ over $K$.
Recall that $\mathbf{E}_{4,\gamma}(K)$ contains a point of order 4 and
$$j(\mathbf{E}_{4,\gamma})=\frac{1}{\gamma^4}=\left(\frac{1}{\gamma}\right)^4=\delta^4=c.$$
This proves the second assertion.  In order to prove the uniqueness, let us assume that
 $E$ is an elliptic curve over $K$ such that $j(E)=c$ and $E(K)$ contains a point of order 4. Then, thanks to Corollary \ref{gamma4},
there exists nonzero $\gamma \in K$ such that $E$ is $K$-isomorphic to $\mathcal{E}_{4,\gamma}$. This implies that
$$c=j(E)=j(\mathbf{E}_{4,\gamma})=\frac{1}{\gamma^4}.$$
It follows that $\gamma=\sqrt[4]{1/c}$.
\end{proof}

\begin{thm}
\label{order82}
An  elliptic curve  over a field $K$ of characteristic $2$ has a $K$-rational point of order $8$ if and only if there exists
$ t \in K\setminus \{0,1\}$ such that
$E$ is $K$-isomorphic to the elliptic curve
$$\mathbf{E}_{8,t}: y^2+xy=x^3+\left(\frac{t}{t^2+1}\right)^8.$$
\end{thm}

\begin{proof}
If an elliptic curve $E$ over $K$ has a $K$-point of order $8$, then it has a $K$-point of order $4$,
and therefore by Corollary \ref{gamma4} is $K$-isomorphic to
$$\mathbf{E}_{4,\gamma}: y^2+xy=x^3+\gamma^4,$$
where $\gamma$ is a nonzero element of $K$. A point $Q$ of order 4 in $E(K)$ has $x$-coordinate
$\gamma \ne 0$. Since $Q$ is divisible by $2$ in $E(K)$, Theorem \ref{division22} tells us that there exist
$l,r \in K$ such that
$$l^2+l=\gamma=r^2.$$
Now the parametrization
$$l=\frac{1}{t^2+1}, \ r= \frac{t}{t^2+1}$$
of the conic
$$l^2+l=r^2$$
gives us the formula
$$\gamma=\left(\frac{t}{t^2+1}\right)^2.$$
\end{proof}

\begin{rem}
\label{jE8}
Let $t,s$ be two nonzero elements of $F$.
Then $$t+\frac{1}{t}=s+\frac{1}{s}$$
 if and only if $s=t$ or $s=t^{-1}$.
 %Since
 %$$\left(\frac{t}{t^2+1}\right)^8=(t+\frac{1}{t})^8, \  \left(\frac{s}{s^2+1}\right)^8=(s+\frac{1}{ts})^8.$$
 This implies that $\mathbf{E}_{8,t}=\mathbf{E}_{8,s}$ if and only if $s=t^{\pm 1}$. (Hereafter both $s$ and $t$ are different from $1$.)
 On the other hand,
 $$j(\mathbf{E}_{8,t})=\left(t+\frac{1}{t}\right)^{-8}, \  j(\mathbf{E}_{8,}s)=\left(s+\frac{1}{s}\right)^{-8}.$$
 This implies that  $\mathbf{E}_{8,t}$ and $\mathbf{E}_{8,s}$ are isomorphic over $\bar{K}$ if and only if $s=t^{\pm 1}$,
 i.e., if and only if $\mathbf{E}_{8,t}=\mathbf{E}_{8,s}$.
\end{rem}

\begin{rem}
\label{ch82}
Let us find explicitly a point of order 8 in $\mathbf{E}_{8,t}(K)$.
We have the order 4 point
$$P=(x_0,y_0)=(\gamma, \gamma^2) \in \mathbf{E}_{8,t}(K)$$
and the equalities
$$l=\frac{1}{t^2+1}, \ r= \frac{t}{t^2+1},  \ x_0=\gamma=\left(\frac{t}{t^2+1}\right)^2=l^2+l=r^2, \ y_0=\gamma^2=\left(\frac{t}{t^2+1}\right)^4.$$
Let $Q=(x_1,y_1)\in \mathbf{E}_{8,t}(K)$ satisfy $2Q=P$ and $l$ be the slope of the tangent line $y=lx+m$ to $E$ at $Q$. Then
$$(y_0+x_0)-l x_0=m=x_1^2.$$
This implies that
$$(\gamma^2+\gamma)-\frac{\gamma}{t^2+1}=m =x_1^2,$$
i.e.,
$$x_1=(\gamma+\sqrt{\gamma})+\frac{\sqrt{\gamma}}{t+1}=\left(\frac{t}{t^2+1}\right)^2+\frac{t}{t^2+1}+\frac{t}{(t+1)^3}=$$
$$\frac{t^2+t(t^2+1)+t(t+1)  }{(t+1)^4}=\frac{t^3}{(t+1)^4}.$$
We have
$$m=x_1^2=\frac{t^6}{(t+1)^8}, \ y_1=l x_1+m=\frac{1}{t^2+1} \frac{t^3}{(t+1)^4}+\frac{t^6}{(t+1)^8}=\frac{t^3}{(t+1)^6}+\frac{t^6}{(t+1)^8}=$$
$$\frac{t^3(t^2+1)+t^6}{(t+1)^8}.$$
To summarize:
$$Q=\left(\frac{t^3}{(t+1)^4},\frac{t^6+t^5+t^3}{(t+1)^8}\right) \in \mathcal{E}_{8,t}(K)$$
is a point of order 8.
\end{rem}

\begin{cor}
Let $q$ be a power of $2$ and $\mathbb{F}_q$ a finite field that consists of $q$ elements.
Let $N$ be a power of $2$ and $\Sigma(q,N)$ the set of $\mathbb{F}_q$-isomorphism classes of elliptic curves $E$ over $\mathbb{F}_q$
such that $E(\mathbb{F}_q)$ contains a point of order $N$.  Then
$$|\Sigma(q,4)|=q-1, \ |\Sigma(q,8)|=\frac{q}{2}-1.$$
\end{cor}

\begin{proof}
We have
$j(\mathbf{E}_{4,\gamma})=1/\gamma^4$. Since $\gamma$ may take any nonzero value in $\mathbb{F}_q$,  it follows from Corollary \ref{gamma4} that
$|\Sigma(q,4)|=q-1.$

It follows from Theorem \ref{order82} combined with Remark \ref{jE8} that there is a bijection between $\Sigma(q,8)$ and the set of all unordered pairs
$\{(t,t^{-1})\mid t \in \mathbb{F}_q\setminus \{0,1\}\}$.  Since $t \ne t^{-1}$ for each $t \ne 0,1$, the set of such pairs consists of
  $(q-2)/2$ elements.
\end{proof}

\begin{ex}
Let us take $q=4$ and $K=\mathbb{F}_4$. Then there is exactly one elliptic curve $E$ over $\mathbb{F}_4$ (up to an $\bar{\mathbb{F}}_4$-isomorphism)
such that $E(\mathbb{F}_4)$ contains a point of order 4, namely,
$$\mathbf{E}_{4,\rho}=\mathbf{E}_{4,\rho+1}:y^2+xy=x^3+1,$$
where $\rho \in \mathbb{F}_4\setminus \mathbb{F}_2$ satisfies
$$\rho^2+\rho+1=0, \rho^3=(\rho+1)^3=1,  \rho^{-1}=\rho+1,  \ \rho+\rho^{-1}=1.$$
The group $\mathbf{E}_{4,\rho}(\mathbb{F}_4)$ contains a point $Q$ of order 8, namely (Remark \ref{ch82}),
$$Q=\left(\frac{\rho^3}{(\rho+1)^4},\frac{\rho^6+\rho^5+\rho^3}{(\rho+1)^8}\right)=(\rho, \rho).$$
This implies that  the order of the finite group $\mathbf{E}_{4,\rho}(\mathbb{F}_4)$ is divisible by 8. On the other hand,
the Hasse bound tells us that the order of $\mathbf{E}_{4,\rho}(\mathbb{F}_4)$ does not exceed $4+2\sqrt{4}+1<8\cdot 2$.
This implies that $\mathbf{E}_{4,\rho}(\mathbb{F}_4)$ has order 8 and therefore coincides with its cyclic subgroup of order 8 generated by $Q$.
\end{ex}

\end{document}